\documentclass[reqno]{amsart}

\numberwithin{equation}{section} 
\usepackage{amsmath,amsfonts,amssymb,amsthm,latexsym}
\usepackage{enumerate}
\usepackage{cancel}
\usepackage{stackrel}

\usepackage{cases}
\usepackage[square,comma,numbers,sort&compress]{natbib}

\usepackage[colorlinks=true]{hyperref}
\hypersetup{urlcolor=blue, citecolor=red}

\newcounter{mnote}
\theoremstyle{plain}
\newtheorem{theorem}{Theorem}[section]
\newtheorem{proposition}[theorem]{Proposition}
\newtheorem{lemma}[theorem]{Lemma}
\newtheorem{corollary}[theorem]{Corollary}
\newtheorem{claim}{Claim}[section]

\theoremstyle{definition}
\newtheorem{definition}[theorem]{Definition}

\theoremstyle{remark}

\newcommand{\vect}[1]{\mathbf{#1}}

\newcommand{\bk}{\vect{k}}
\newcommand{\bu}{\vect{u}}
\newcommand{\bv}{\vect{v}}

\newcommand{\bx}{\vect{x}}
\newcommand{\by}{\vect{y}}

\newcommand{\bm}{\vect{m}}

\newcommand{\field}[1]{\mathbb{#1}}

\newcommand{\nT}{\field{T}}
\newcommand{\nZ}{\field{Z}}

\newcommand{\nR}{\field{R}}

\newcommand{\nL}{\field{L}}

\newcommand{\vphi}{\varphi}
\newcommand{\veps}{\varepsilon}
\newcommand{\maps}{\rightarrow}

\newcommand{\pd}[2]{\frac{\partial #1}{\partial #2}}
\newcommand{\od}[2]{\frac{d #1}{d #2}}
\newcommand{\npd}[3]{\frac{\partial^#3 #1}{\partial #2^#3}}
\newcommand{\abs}[1]{\left\lvert#1\right\rvert}
\newcommand{\norm}[1]{\left\lVert#1\right\rVert}
\newcommand{\set}[1]{\left\{#1\right\}}

\newcommand{\grad}{\text{$\nabla$}}

\newcommand{\lap}{\Delta}
\newcommand{\strong}{\rightarrow}

\newcommand{\weakstar}{\stackrel[\ast]{}{\rightharpoonup}}
\newcommand{\convh}{\stackrel[h]{}{\ast}}

\newcommand{\LpP}[1]{\text{$L^{#1}(\nT^3)$}}

\newcommand{\nLpP}[1]{\text{$\nL^{#1}(\nT^3)$}}

\newcommand{\LpR}[2]{\text{$L^{#1}({\nR}^{#2})$}}

\newcommand{\WP}[2]{\text{$W^{#1,#2}(\nT^3)$}}

\newcommand{\CinftyR}[1]{\text{$C^{\infty}(\nR^{#1})$}}

\newcommand{\Lph}[1]{\text{$L^{#1}_h(\Omega)$}}
\newcommand{\LphP}[1]{\text{$L^{#1}_h(\nT^3)$}}
\newcommand{\nLphP}[1]{\text{${\nL}^{#1}_h(\nT^3)$}}

\newcommand{\LzP}[1]{\text{$L^{#1}_z(\nT^3)$}}

\newcommand{\WhP}[2]{\text{$W^{#1,#2}_h(\nT^2)$}}

\newcommand{\LzLphP}[2]{\text{$L^{#1,#2}_{z,h}(\nT^3)$}}
\newcommand{\nLzLphP}[2]{\text{${\nL}^{#1,#2}_{z,h}(\nT^3)$}}

\newcommand{\LzWhP}[3]{\text{$L^{#1}_z(W^{#2,#3}_h(\nT^3))$}}

\begin{document}
\title[3D Inviscid pseudo-Hasegawa-Mima Model]{Global Well-posedness of an inviscid three-dimensional pseudo-Hasegawa-Mima model}

\date{December 7, 2011.}

\author{Chongsheng Cao}
\address[Chongsheng Cao]{Department of Mathematics\\
                Florida International University\\
                Miami, FL 33199, USA}
\email[Chongsheng Cao]{caoc@fiu.edu}
\author{Aseel Farhat}
\address[Aseel Farhat]{Department of Mathematics\\
                University of California\\
        Irvine, CA 92697-3875, USA}
\email[Aseel Farhat]{afarhat@math.uci.edu}
\author{Edriss S. Titi}
\address[Edriss S. Titi]{Department of Mathematics, and Department of Mechanical and Aero-space Engineering\\
University of California\\
Irvine, CA 92697-3875, USA.
Also, The Department of Computer Science and Applied Mathematics\\
The Weizmann Institute of Science, Rehovot 76100, Israel.
Fellow of the Center of Smart Interfaces (CSI), Technische Universit\"{a}t Darmstadt, Germany. }
\email[Edriss S. Titi]{etiti@math.uci.edu and edriss.titi@weizmann.ac.il}

\begin{abstract}
The three-dimensional inviscid Hasegawa-Mima model is one of the fundamental models that describe plasma turbulence. The model also appears as a simplified reduced Rayleigh-B\'enard convection model. The mathematical analysis the Hasegawa-Mima equation is challenging  due to the absence of any smoothing viscous terms, as well as to the presence of an analogue of the vortex stretching terms. In this paper, we introduce and study a model which is inspired by the inviscid Hasegawa-Mima model, which we call a pseudo-Hasegawa-Mima model. The introduced model is easier to investigate analytically than the original inviscid Hasegawa-Mima model, as it has a nicer mathematical structure. The resemblance between this model and the Euler equations of inviscid incompressible fluids inspired us to adapt the techniques and ideas introduced for the two-dimensional and the three-dimensional Euler equations to prove the global existence and uniqueness of solutions for our model. Moreover, we prove the continuous dependence on initial data of solutions for the pseudo-Hasegawa-Mima model. These are the first results on existence and uniqueness of solutions for a model that is related to the three-dimensional inviscid Hasegawa-Mima equations.

\end{abstract}

 \maketitle

 

\section{Introduction}

The three-dimensional inviscid Hasegawa-Mima equations were first derived by Hasegawa and Mima as a simple model that describes plasma turbulence. The Hasegawa-Mima equations are also one of the simplest and most fundamental models that describe the electrostatic drift waves. The three-dimensional equations that describe the coupling of the drift modes to ion-acoustic waves that propagate along the magnetic field are given by the system (c.f., \cite{Akerstedt_Nycander_Pavlenko_1996}, \cite{Hasegawa_Mima_1978}, \cite{Horton_Meiss_1983}, \cite{Guo_Zhang_2007}):
\begin{subequations}\label{3D_Hasegawa_Mima}
\begin{align}
\frac{\partial}{\partial t} (\triangle_h\phi- \phi) +J(\phi,\triangle_h\phi) + v_d\frac{\partial \phi}{\partial y} - \frac{\partial v_z}{\partial z} & =0 ,\\
\frac{\partial v_z}{\partial t} + J(\phi,v_z) + \frac{\partial \phi }{\partial z} &= 0,
\end{align}
\end{subequations}
where $J$ is the two-dimensional Jacobian defined by $J(f,g) = \frac{\partial f}{\partial x}\frac{\partial g}{\partial y} - \frac{\partial f}{\partial y}\frac{\partial g}{\partial x}$ and $\triangle_h = \frac{\partial^2}{\partial x^2}+\frac{\partial^2}{\partial y^2}$ is the horizontal Laplacian. Here $\phi$ is the electrostatic potential and is simultaneously the stream function for the horizontal flow in the $xy$-plane, so that the horizontal velocity is given by ${\bf v}_h = \bf{\hat{z}} \times \grad_h\phi$. Here $v_d$ is a constant that is proportional to the density gradient and $v_z$ is the normalized (by the ion-acoustic wave speed) velocity in the $z$-direction. Notice that the last term in the first equation is the ``ghost" of  the vortex stretching.

In the context of geophysical fluid dynamics, the three-dimensional inviscid Hasegawa-Mima equations appear as a reduced Rayleigh-B$\acute{e}$nard convection model which describes the flow motion of a fluid heated from below. The three-dimensional rotationally constrained convection model was derived in \cite{Julien_Knobloch_Milliff_Werne_2006}, and \cite{Sprague_Julien_Knobloch_Weme_2006}, under the assumption of very small Rossby number. In the three-dimensional domain $\Omega = [0,L]^3$, the equations are:
\begin{subequations}\label{Rayleigh_Benard}
\begin{align}
\frac{\partial w}{\partial t} + ({\bold{u}} \cdot \nabla_h) w - \frac{\partial ((-\triangle_h)^{-1}\omega)}{\partial z} & = \Gamma \Theta^{'} +\frac{1}{Re}\triangle_h w, \\
\frac{\partial \omega}{\partial t} +({\bold{u}}\cdot \nabla_h) \omega -\frac{\partial w}{\partial z} & = \frac{1}{Re}\triangle_h\omega,\\
\frac{\partial {\Theta^{'}}}{\partial t} + ({\bold{u}} \cdot \nabla_h) \Theta^{'} +w \frac{\partial {\bar \Theta}}{\partial z} &= \frac{1}{Pe}\triangle_h \Theta^{'}, \\
-\frac{1}{Pe} \frac{\partial^2 {\bar\Theta}}{\partial z^2} + \frac{\partial (\overline{\Theta^{'} w})}{\partial z} &=0,\\
\int _{[0,L]^2} \Theta^{'} \ dxdy &=0,\\
\nabla_h \cdot {\bold{u}} &= 0,
\end{align}
\end{subequations}
where the horizontal-average of a given function is defined as:
\begin{equation*}
{\bar \phi} = \frac{1}{L^2}\int_{[0,L]^2} \phi \ dxdy.
\end{equation*}
Here, ${\bold{u}} = (u,v)$ is the horizontal component of the velocity vector field $(u,v,w)$, and the vertical component of the vorticity, $\omega = \nabla_h\times{\bold{u}} := \frac{\partial v}{\partial x}-\frac{\partial u}{\partial y}$, and the fluctuation of the temperature, $\Theta^{'}= \Theta - { \bar \Theta}$, are unknown functions of $(t;x,y,z)$, while the horizontal-mean temperature ${\bar \Theta}$ is an unknown function of $(t;z)$. $Re$ is the Reynolds number, $\Gamma$ is the buoyancy number, $Pe$ is the P$\acute{e}$clet number, $\nabla_h =(\frac{\partial}{\partial x}, \frac{\partial}{\partial y})$ and $\triangle_h = \frac{\partial^2}{\partial x^2}+\frac{\partial^2}{\partial y^2}$.

The two-dimensional Hasegawa-Mima equations were derived by Charney and Obukhov as a shallow water model from the Euler equations with free surface under quasi-geostrophic balance assumption. The geophysical model that was obtained is:
\begin{equation}\label{2D_Hasigawa_Mima}
\pd{(\triangle_h\phi_0 - F\phi_0)}{t}+ J(\phi_0,\triangle_h\phi_0) + J(\phi_0,\phi_B+\beta y) = 0,
\end{equation}
where $\phi_0(x,y)$ is the amplitude of the surface perturbation at the lowest order in the Rossby number, $\phi_B(x)$ is a given function that parametrizes the bottom topography, and here again $J$ is the Jacobian while $F$ is the Froude number constant. When a flat bottom is considered; that is when $\phi_B$ is taken to be equal to constant one obtains
\begin{equation}\label{HMCO}
\pd{(\triangle_h\phi_0-F\phi_0)}{t}+ J(\phi_0,\triangle_h\phi_0) + \beta\pd{\phi_0}{x} = 0,
\end{equation}
which is the inviscid Hasegawa-Mima-Charney-Obukhov (HMCO) equation, where $\phi_0 = \phi_0(t;x,y)$. Global existence and uniqueness of solutions of the inviscid HMCO equation can be easily established following the same ideas as for the two-dimensional Euler equations. Indeed, it was remarked in \cite{Gao_Zhu_2005}, and in \cite{Paumond_2004}, that \eqref{HMCO}, with $F=\beta=0$, is equivalent to the two-dimensional Euler equations of incompressible inviscid fluid, where $\phi_0$ plays the role of the stream function, $\omega = \triangle_h\phi_0$ the vorticity and $u = (-\pd{\phi_0}{y},\pd{\phi_0}{x})$ is the velocity field. The existence and uniqueness of strong local solutions of $\eqref{HMCO}$ in $H^s(\nR^2)$, with initial data $\phi_0 \in H^s(\nR^2)$, for $s\geq 4$; and the existence of a weak global solution in $H^2(\nR^2)$, with initial data $\phi_0 \in H^2(\nR^2)$, were established in \cite{Paumond_2004}.  The uniqueness of a global strong solution in $H^s(\nR^2)$, for $s\geq4$, was later established in \cite{Gao_Zhu_2005}. In a recent paper \cite{Farhat_Hauk_Titi_2011}, we followed the work of Yudovich \cite{Yudovich_1963} to prove the existence and uniqueness for the vorticity $\triangle_h \phi_0$ in the space $L^\infty$, in the case of no-normal flow boundary conditions, i.e. Dirichlet boundary conditions for the stream function $\phi_0$.  Notice that $\phi_0\in H^s$, for $s\geq 4$, is equivalent to $\triangle_h \phi_0\in H^{s-2}$. Thus, our result, reported in \cite{Farhat_Hauk_Titi_2011}, improves the previous results, since $\triangle_h \phi_0 \in H^{s-2} \subset L^\infty$ for $s>3$.

The two-dimensional viscous version of the Hasegawa-Mima equations was also studied, later on, by some authors using standard tools used for studying the two-dimensional Navier-Stokes equations. For example, the existence of a global attractor in $H^2(\Omega)$ of the viscous two-dimensional equations, was studied in \cite{R_Zhang_2008} and in \cite{Guo_Zhang_2006}; and an upper bounds for its Hausdorff and fractal dimensions were also established. A perturbed version of the two-dimensional Hasegawa-Mima equations was studied in \cite{Grauer_1998}, the authors considered the two-dimensional Hasegawa-Mima equations with a hyper-viscosity term and used techniques introduced for the Kuramoto-Sivashinsky equation (KS) by Nicolaenko, Scheurer and Temam \cite{Nicolaenko-Scheurer-Temam_1985}, for odd-periodic solutions. The authors in \cite{Guo_Zhang_2007} studied the existence of a global attractor of a viscous three-dimensional model similar to the three-dimensional Hasegawa-Mima equations that they called  a ``generalized" version of the three-dimensional Hasegawa-Mima equations and they established upper bounds for its dimension. We emphasize that here the authors followed well-established tools for dissipative  equations (see, e.g., \cite{Temam_1997} and reference therein) to prove the existence of the global attractor and estimating its dimension. The added artificial three-dimensional viscosity simplifies the mathematical analysis drastically and allows for the implementation of these well-established tools. We stress here that all the previous studies have been focused on the viscous and the inviscid two-dimensional Hasegawa-Mima equations, and on the viscous three-dimensional case following well-established tools.

Investigating the global existence of \eqref{Rayleigh_Benard} is challenging due to the fact that the unknown functions depend on the three spatial variables $(x,y,z)$, while the regularizing viscosity terms acts only in the horizontal variables $(x,y)$. The situation for \eqref{3D_Hasegawa_Mima} is even more challenging since it has no viscous regularization at all. Inspired by these models, i.e. \eqref{3D_Hasegawa_Mima} and \eqref{Rayleigh_Benard}, we introduce the following ``simplified" mathematical inviscid model -- the pseudo-Hasegawa-Mima equations:
\begin{subequations}\label{simplified_3D_Hasegawa_Mima}
\begin{align}
\frac{\partial w}{\partial t} + ({\bold{u}} \cdot \nabla_h) w- U_0L\frac{\partial \omega}{\partial z}& = 0, \label{simplified_3d_HM_a}\\
\frac{\partial \omega}{\partial t} +({\bold{u}}\cdot \nabla_h) \omega -\frac {U_0}{L}\frac{\partial w}{\partial z}& = 0,\label{simplified_3d_HM_b}\\
\nabla_h \cdot \bold{u} &= 0,
\end{align}
\end{subequations}
for some $U_0$ velocity scale in the $z$-direction. Here ${\bold{u}}=(u,v)$ is the horizontal component of the velocity vector field $(u,v,w)$, and the vorticity $\omega = \nabla_h\times{\bold{u}} := \frac{\partial v}{\partial x}-\frac{\partial u}{\partial y}$. One can see easily that \eqref{simplified_3D_Hasegawa_Mima} is a simplified version of the Hasigawa-Mima equations \eqref{3D_Hasegawa_Mima}. The goal of introducing, and investigating, this mathematical model is to shed light on the analysis of \eqref{3D_Hasegawa_Mima} and \eqref{Rayleigh_Benard}. The system \eqref{simplified_3D_Hasegawa_Mima} is simpler because it has a nicer mathematical structure. Specifically, let us denote:
\begin{align}
\theta = w + L\omega, \quad \quad \eta = w - L\omega.
\end{align}
Adding and subtracting \eqref{simplified_3d_HM_a} and \eqref{simplified_3d_HM_b} yield the coupled system:
\begin{align}\label{simplified_3d_HM_theta_eta}
\pd{\theta}{t} + (\bu\cdot \nabla_h)\theta - U_0\pd{\theta}{z} = 0, \quad \pd{\eta}{t} +(\bu\cdot\nabla_h)\eta + U_0\pd{\eta}{z} = 0, \quad \nabla_h\cdot \bu &=0,
\end{align}
where $\nabla_h\times\bu=\omega=\frac{1}{2L}(\theta-\eta)$. System \eqref{simplified_3d_HM_theta_eta} is similar in its structure to the two-dimensional Euler equations due to the relation between $\bu$ and the vorticity $\omega$, where the variable $z$ is thought of as a parameter, and also in the sense that it is a purely three-dimensional transport system that does not have a vorticity stretching term. Notice that the velocity fields that transport $\theta$ and $\eta$ in \eqref{simplified_3d_HM_theta_eta} are not the same. This is in fact a reflection of the vortex stretching terms in the original model \eqref{simplified_3D_Hasegawa_Mima}.

We denote by $\nT^d$ the L-period box $[0,L]^d$. We will establish in this work the global existence and uniqueness for system \eqref{simplified_3d_HM_theta_eta} in $\nT^3$, i.e. subject to periodic boundary conditions, and then use this result to conclude the global existence and uniqueness of solutions for system \eqref{simplified_3D_Hasegawa_Mima} in $\nT^3$. 
As in the case of the two-dimensional Euler equations in vorticity formulation (see, e.g., \cite{Majda_Bertozzi_2002} and \cite{Marchioro_Pulvirenti_1994}), we have the analogue  of the two-dimensional periodic Biot-Savart law:
\begin{subequations}\label{conv}
\begin{align}
\bu(x,y,z) = K\convh\omega &:=\int_{\nT^2} K(x-s,y-\xi)\,\omega(s,\xi,z)\,dsd\xi,\\
K(x,y) &= \nabla_h^\perp G(x,y),
\end{align}
\end{subequations}
where $G(x,y)$ is the fundamental solution of the Poisson equation in two-dimensions subject to periodic boundary conditions, the binary operation $\convh$ denotes the horizontal convolution, and $\nabla_h^\perp = (-\pd{}{y}, \pd{}{x})$. In \eqref{conv} we explicitly restrict ourselves to the unique solution $\bu$ of the elliptic system: $\nabla_h\times\bu=\omega$ and  $\nabla_h\cdot \bu =0$, that satisfies the side condition $\int_{\nT^2} \bu(x,y,z) \,dx\,dy = 0$ for every $z \in [0,L]$.

In section $2$, we state and prove a technical two-dimensional logarithmic inequality that improves the Brezis-Gallouet \cite{Brezis_Gallouet_1980} and the Brezis-Wainger \cite{Brezis_Wainger_1980} logarithmic Sobolev inequalities. This logarithmic inequality allows us to prove some useful estimates that we use in proving the global existence and uniqueness results. In section $3$, we will study the global existence and uniqueness of smooth solutions of a ``linearized" version of  \eqref{simplified_3d_HM_theta_eta} and establish some relevant uniform estimates. In section $4$, we use the established results and the uniform estimates obtained for the ``linearized" system, in section 3, to prove a global existence and uniqueness result of solutions, and obtain some uniform estimates for \eqref{simplified_3d_HM_theta_eta}. Even though system \eqref{simplified_3d_HM_theta_eta} does not seem to have explicit vorticity stretching term, but the terms $\pd{\theta}{z}$ and $\pd{\eta}{z}$  can be thought of as if they are vorticity stretching terms and produce similar challenge. The main challenge in this work is in obtaining the relevant uniform bounds that we proved based on estimates we obtain in section $2$, using the logarithmic Sobolev-type inequality. We stress that our results are one of the first results related to the inviscid three-dimensional Hasegawa-Mima equations.

\bigskip
\section{Preliminaries}
In this section, we introduce some preliminary material and notations which are commonly used in the mathematical study of fluids, in particular in the study of the Navier-Stokes equations (NSE) and the Euler equations.

Let $\mathcal{F}$ be the set of all trigonometric polynomials with periodic domain $\nT^3:=[0,L]^3$.
We denote by $\LpP{p}$, $\WP{s}{p}$, ${H}^s(\nT^3)\equiv \WP{s}{2}$ to be the closures of $\mathcal{F}$ in the usual Lebesgue and Sobolev spaces, respectively.  Since we restrict ourselves to finding solutions over the three-dimensional L-periodic box $\nT^3$, therefore, we can work in the spaces defined above consistently. We define the inner product on ${L}^2(\nT^3)$:
\[\left<\bu,\bv\right>= \int_{\nT^3} \bu(\bx)\cdot \bv(\bx)\,d\bx,
\]
and the associated norm $\norm{\bu}_{\LpP{2}}=\left<\bu,\bu\right>^{1/2}$.  (We use this notation for both scalar and vector fields). We define the following Banach spaces:
$$L_z^{p}(\nT^3):= \set{f: \norm{f(x,y)}_{L_z^p(\nT^3)}^p:=\int_{0}^{L}\abs{f(x,y,z)}^p \,dz < \infty,\,\text{for every}\, (x,y) \in [0,L]^2},
$$
$$L_h^{p}(\nT^3):= \set{f: \norm{f(z)}_{L_h^{p}(\nT^3)}^p:=\int_{0}^{L}\int_{0}^{L}\abs{f(x,y,z)}^p \,dxdy <\infty,\,\text{for every}\, z\in [0,L]},
$$
$$L_h^{\infty}(\nT^3):= \set{f: \norm{f(z)}_{L_h^{\infty}(\nT^3)}:= \stackrel[{(x,y)\in[0,L]^2}]{}{\text{esssup}}\abs{f(x,y,x)} <\infty,\,\text{for every}\, z\in [0,L]},
$$
$$\LzLphP{\infty}{p}:= \set{f: \norm{f}_{L_z^\infty(L_h^p(\nT^3))}:= \stackrel[{z\in[0,L]}]{}{\text{esssup}}\norm{f(z)}_{L_h^{p}(\nT^3)} <\infty}.
$$
Similarly, we define the spaces $\LzWhP{\infty}{s}{p}$ and $L_z^{\infty}(H^s_h(\nT^3)$. We also denote by the Banach space:
$${\nL}^p(\nT^3):= {L}^p(\nT^3)\times{L}^p(\nT^3)$$
and its associated norm:
$$\norm{(\phi,\psi)}_{{\nL}^p(\nT^3)} = \norm{\phi}_{{L}^p(\nT^3) }+ \norm{\psi}_{{L}^p(\nT^3)}.$$
Similarly, we define the spaces $\nL_z^{p}(\nT^3), \nL_h^{p}(\nT^3), \nL_h^{\infty}(\nT^3)$ and $\nLzLphP{\infty}{p}$. The natural definitions of $\norm{.}_{\WP{m}{p}}$ and $\norm{.}_{\WhP{m}{p}}$ are
\begin{align}
\norm{\bu}_{\WP{m}{p}}^2 := \sum_{\abs{\alpha}\leq m} L^{\frac{2(\abs{\alpha}-m)}{3}}\int_{\nT^3}\abs{D^{\alpha}\bu(x,y,z)}^2\,dx\,dy\,dz,
\end{align}
and
\begin{align}
\norm{\bu(z)}_{\WhP{m}{p}}^2 := \sum_{\abs{\alpha}\leq m} L^{\abs{\alpha}-m}\int_{\nT^2}\abs{D^{\alpha}\bu(x,y,z )}^2\,dx\,dy.
\end{align}

Let $Y$ be a Banach space.  We denote by $L^p([0,T];Y)$ the space of (Bochner) measurable functions $t\mapsto w(t)$, where $w(t)\in Y$, for a.e. $t\in[0,T]$, such that the integral $\int_0^T\|w(t)\|_Y^p\,dt$ is finite. A similar convention is used in the notation $C^k([0,T];Y)$ for $k$-times differentiable functions of time on the interval $[0,T]$ with values in $Y$. We denote by ${C}^\infty_{per}([0,T];\nT^3)$ the set of infinitely differentiable functions in the variable $\bx$ and $t$ which are periodic in $\bx$, and we denote by $C^k_{per}(\nT^3)$ the set of k-times differentiable periodic functions in the variable $\bx$ with periodic domain $\nT^3$.
Next, we recall the Ladyzhenskaya inequality for an integrable function $\vphi$ in two dimensions,
\begin{equation}\label{L4_to_H1}
\|\vphi(z)\|_{\LphP{4}}\leq C \|\vphi(z)\|_{\LphP{2}}^{1/2}\|\vphi(z)\|_{H^1_h}^{1/2},
\end{equation}
for every $z\in[0,L]$. Hereafter, $C$ denotes a generic constant which may change from line to line.
Furthermore, for all $\bu \in H^1_h(\nT^2)$, with $\int_{\nT^2}\bu(x,y,z) \,dx\,dy = 0$, for every $z\in[0,L]$, we have the Poincar\'e inequality
\begin{align}\label{Poincare_inequality}
\norm{\bu(z)}_{L^2_h(\nT^2)} \leq CL \norm{\bu(z)}_{H^1_h(\nT^2)}.
\end{align}
We recall the well-known elliptic estimate in two-dimensions, due to the Biot-Savart law for an incompressible vector field $\bu$, satisfying $\nabla_h\cdot \bu=0$ and $\nabla_h\times\bu = \omega$,
\begin{equation}\label{elliptic_regularity_Yudovich}
\norm{\bu(z)}_{W^{1,p}_h} \leq C p \|\omega(z)\|_{L^p_h},
\end{equation}
for every $z\in[0,L]$ and every $p\in [2,\infty)$ (see, e.g., \cite{Yudovich_1963} and references therein).
Also, it is an easy exercise to show using the divergence free condition, $\nabla_h\cdot \bu=0$, that
\begin{align}
\label{nabla_u_omega_L2}
\norm{\omega(z)}_{L^p_h} = \norm{\nabla_h\times\bu(z)}_{L^p_h} = \sqrt{2} \norm{\nabla_h\bu(z)}_{L^p_h}.
\end{align}

Next, we introduce a generalization of the logarithmic inequality that was introduced in \cite {Cao_Wu_2011}. Similar inequalities were previously introduced by other authors in \cite{Danchin_Paicu_2011}, \cite{Kupferman_Mangoubi_Titi_2008} and \cite{Zhang_2008}. These inequalities improve the Brezis-Gallouet \cite{Brezis_Gallouet_1980} and the Brezis-Wainger \cite{Brezis_Wainger_1980} logarithmic Sobolev inequalities. We follow the same idea of the proof given by the authors in \cite{Cao_Wu_2011} and \cite{Hou_Li_2005}. We give the proof in $\nR^2$, but one can derive the same result for periodic functions (see Corollary \ref{periodic_logarithmic_inequality} below).

\begin{lemma}\label{logarithmic_inequality_3D_HM}
Let $F: \nR^2 \maps \nR$ be a scalar valued function, $F \in W^{1,2+\delta}(\nR^2)$, and $\delta>0$ be given. Then, for any $\lambda \in (0,\infty)$ and any radius $R \in(0,\infty)$, the following logarithmic inequality holds:
\begin{align}\label{logarithmic_inequality_3D_HM_stated}
\norm{F}_{\LpR{\infty}{2}} &\leq C_{\delta,\lambda}\max\set{1,\sup_{q\geq2} \frac{\norm{F}_{\LpR{q}{2}}}{q^{\lambda} R^{2/q}}} \notag \\
&\qquad \log^{\lambda}\left(e^2 + C R^{\delta/(2+\delta)}\frac{\norm{F}_{\LpR{2+\delta}{2}}}{R}+ \norm{\nabla F}_{\LpR{2+\delta}{2}}\right).
\end{align}
where $C_{\delta,\lambda}$ is an absolute positive constant that depends on $\delta$ and $\lambda$.
\end{lemma}
\begin{proof}
Without loss of generality, we assume that ${\norm{F}}_{\LpR{\infty}{2}} = F({\bf{0}})\not=0.$ Denote by $B_R$ the disk centered at the origin with radius R. Let $\phi \in \CinftyR{2}$ be a smooth non-negative cutoff function satisfying:
 \begin{align}
 \phi({\bf0}) = 1, \qquad \norm{\phi}_{\LpR{\infty}{2}} \leq 1, \qquad  \norm{\grad \phi}_{\LpR{\infty}{2}} \leq \frac{C}{R}, \qquad \text{supp}(\phi) \subset B_R.
 \end{align}
Set $\zeta = F\phi$. From the solution formula of the Laplace equation in $\nR^2$ we have, for any $p\geq2$,
 \begin{align*}
 \abs{\zeta^p({\bf{0}})} = \abs{\frac{1}{2\pi} \int_{\nR^2} \log(\abs{\by}) \lap \zeta^p({\bf y})\, d\by|} = \abs{\frac{1}{2\pi} \int_{B_R} \log(\abs{\by}) \lap\zeta^p({\bf y})\, d\by|}.
 \end{align*}
By integration by parts,
\begin{align*}
 \abs{\zeta^p({\bf 0})} & = \abs{\frac{p}{2\pi} \int_{B_R} \frac{\by\cdot\grad \zeta}{{\abs{\by}}^2}\zeta^{p-1} \,d\by}\\
 & \leq  \abs{\frac{p}{2\pi} \int_{B_\veps} \frac{\by\cdot\grad \zeta}{{\abs{\by}}^2}\zeta^{p-1} \,d\by}
 + \abs{\frac{p}{2\pi} \int_{B_R\setminus B_\veps} \frac{\by\cdot\grad \zeta}{{\abs{\by}}^2}\zeta^{p-1} \,d\by},
 \end{align*}
for some $\veps \in (0,R)$, to be chosen later in \eqref{choose_eps_logarithmic_inequality_simplified_3D_HM}. By H\"older inequality,
 \begin{align*}
\abs{\zeta^p({\bf0})} \leq Cp& \norm{\frac{\by}{\abs{\by}^2}}_{L^\alpha(B_\veps)} \norm{\grad\zeta}_{L^{2+\delta}(B_\veps)} \norm{\zeta}^{p-1}_{L^{\beta(p-1)}(B_\veps)} \\ &+ Cp  \norm{\frac{\by}{\abs{\by}^2}}_{L^2(B_R\setminus B_\veps)} \norm{\grad\zeta}_{L^{2+\delta}(B_R\setminus B_\veps)} \norm{\zeta}^{p-1}_{L^{\frac{2(2+\delta)}{\delta}(p-1)}(B_L\setminus B_\veps)},
 \end{align*}
where $\alpha$ and $\beta$ satisfy $\frac{1}{\alpha}+\frac{1}{\beta}=\frac{1+\delta}{2+\delta}$, and $\alpha \in(1,2)$ is chosen such that $\norm{\frac{\by}{\abs{\by}^2}}_{L^\alpha(B_\veps)}<\infty$. For instant, one may choose:
 \begin{align}\label{logarithmic_inequality_simplified_3d_HM_alfa_beta}
 \alpha = \frac {4+\delta}{2+\delta} \quad \text{and} \quad \beta = \frac{(2+\delta)(4+\delta)}{\delta}.
 \end{align}
This implies:
 \begin{align}\label{log_inequality_p}
 \abs{\zeta^{p}({\bf0})} \leq C_{\delta}p& \veps^{\frac{2-\alpha}{\alpha}} \norm{\grad \zeta}_{\LpR{2+\delta}{2}} \norm{\zeta}^{p-1}_{\LpR{\beta(p-1)}{2}} \notag\\
 &+ Cp\log^{1/2}\left(\frac{R}{\veps}\right)\norm{\grad\zeta}_{\LpR{2+\delta}{2}} \norm{\zeta}^{p-1}_{\LpR{\frac{2(2+\delta)}{\delta}(p-1)}{2}}.
 \end{align}
 Now, we choose $\veps \in (0,R)$, small enough, such that:
 \begin{align}\label{choose_eps_logarithmic_inequality_simplified_3D_HM}
 \left(\frac{\veps}{R}\right)^{\frac{2-\alpha}{\alpha}}\left(e^2 + R^{\delta/(2+\delta)}\norm{\grad\zeta}_{\LpR{2+\delta}{2}}\right) = 1,
 \end{align}
thus we have,
 \begin{align*}
 \veps^{\frac{2-\alpha}{\alpha}}\norm{\grad\zeta}_{\LpR{2+\delta}{2}}&\leq R^{\frac{2-\alpha}{\alpha}} R^{\frac{-\delta}{2+\delta}},
 \end{align*}
 \begin{align*}
\log\left(\frac{R}{\veps}\right) &= \frac{2-\alpha}{\alpha} \log \left (e^2 + R^{\delta/(2+\delta)}\norm{\grad\zeta}_{\LpR{2+\delta}{2}}\right).
\end{align*}
Recall that $p^{1/p} \leq C$ for every $p\geq2$, thus \eqref{log_inequality_p} and \eqref{choose_eps_logarithmic_inequality_simplified_3D_HM} yield,
 \begin{align}\label{log_inequality_p_2}
&\abs{\zeta({\bf0})} \leq C_{\delta}\left(R^{\frac{2-\alpha}{\alpha}} R^{\frac{-\delta}{2+\delta}}\right)^{1/p}\norm{\zeta}^{1-\frac{1}{p}}_{\LpR{\beta(p-1)}{2}} \\
&+C \left(\frac{2-\alpha}{\alpha}\log\left (e^2 + R^{\delta/(2+\delta)}\norm{\grad\zeta}_{\LpR{2+\delta}{2}}\right)\right)^{\frac{1}{2p}}\norm{\grad\zeta}_{\LpR{2+\delta}{2}}^{1/p}
\norm{\zeta}^{1-\frac{1}{p}}_{\LpR{\frac{2(2+\delta)}{\delta}(p-1)}{2}}.\notag
\end{align}
Observe that since $\beta\geq6$ (see \eqref{logarithmic_inequality_simplified_3d_HM_alfa_beta}) and $p\geq2$, then $\beta(p-1) \geq 2$, $\beta^{\lambda(1-\frac{1}{p})} \leq \beta^\lambda$ and $(p-1)^{\lambda(1-\frac{1}{p})} \leq p^\lambda$. Then,

\begin{align}\label{logarithmic_inequality_3d_simplified_HM_1}
\norm{\zeta}^{1-\frac{1}{p}}_{\LpR{\beta(p-1)}{2}}& = R^{2/(p\beta)}(\beta(p-1))^{\lambda(1-\frac{1}{p})}\left( \frac{\norm{\zeta}_{\LpR{\beta(p-1)}{2}}}{(\beta(p-1))^{\lambda} R^{2/(\beta(p-1))}}\right)^{1-\frac{1}{p}} \notag \\
&\leq R^{2/(p\beta)}(\beta(p-1))^{\lambda(1-\frac{1}{p})}\left( \sup_{q\geq2} \frac{\norm{\zeta}_{\LpR{q}{2}}}{q^{\lambda} R^{2/q}} \right)^{1-\frac{1}{p}} \notag \\
&\leq C R^{2/(p\beta)}(\beta(p-1))^{\lambda(1-\frac{1}{p})} \left(\sup_{q\geq2} \frac{\norm{F}_{\LpR{q}{2}}}{q^{\lambda} R^{2/q}} \right)^{1-\frac{1}{p}} \notag \\
& \leq C_{\delta,\lambda}R^{2/(p\beta)} p^{\lambda} \max \set{1, \sup_{q\geq2} \frac{\norm{F}_{\LpR{q}{2}}}{q^{\lambda} R^{2/q}}}.
\end{align}
Similar argument shows that,
\begin{align}\label{logarithmic_inequality_3d_simplified_HM_2}
\norm{\zeta}^{1-\frac{1}{p}}_{\LpR{\frac{2(2+\delta)}{\delta}(p-1)}{2}} &\leq C_{\delta,\lambda}R^{\delta/(p(2+\delta))} p^{\lambda} \max \set{1, \sup_{q\geq2} \frac{\norm{F}_{\LpR{q}{2}}}{q^{\lambda} R^{2/q}}}.
\end{align}
Since, due to \eqref{logarithmic_inequality_simplified_3d_HM_alfa_beta}, $R^{\frac{2-\alpha}{\alpha}} R^{\frac{-\delta}{2+\delta}}R^{2/\beta}=1$ and $\left(\frac{2-\alpha}{\alpha}\right)^{1/(2p)} \leq1$, \eqref{log_inequality_p_2}, \eqref{logarithmic_inequality_3d_simplified_HM_1} and \eqref{logarithmic_inequality_3d_simplified_HM_2} will yield:
\begin{align}\label{logarithmic_inequality_simplified_3d_HM_2.11}
\abs{\zeta({\bf0})} &\leq C_{\delta,\lambda}\left(1+\log ^{1/(2p)}\left (e^2 + R^{\delta/(2+\delta)}\norm{\grad\zeta}_{\LpR{2+\delta}{2}}\right)\right)\notag \\
& \qquad \qquad \left(R^{\delta/(2+\delta)}\norm{\nabla\zeta}_{\LpR{2+\delta}{2}}\right)^{1/p} p^{\lambda} \max \set{1, \sup_{q\geq2} \frac{\norm{F}_{\LpR{q}{2}}}{q^{\lambda} R^{2/q}}}.
\end{align}
Now, we choose $p\geq2$ such that:
\begin{align}
p = \log\left(e^2+R^{\delta/(2+\delta)}\norm{\nabla \zeta}_{\LpR{2+\delta}{2}}\right),
\end{align}
then we have,
\begin{align}\label{logarithmic_inequality_simplified_3d_HM_2.13}
\left(R^{\delta/(2+\delta)}\norm{\nabla\zeta}_{\LpR{2+\delta}{2}}\right)^{1/p} \leq e.
\end{align}
Thus, we can conclude from \eqref{logarithmic_inequality_simplified_3d_HM_2.11} and \eqref{logarithmic_inequality_simplified_3d_HM_2.13} that
\begin{align}\label{log_inequality_p_3}
\abs{\zeta({\bf0})} &\leq C_{\delta,\lambda}\max \set{1, \sup_{q\geq2} \frac{\norm{F}_{\LpR{q}{2}}}{q^{\lambda} R^{2/q}}}\log^{\lambda}\left(e^2 + R^{\delta/(2+\delta)}\norm{\grad\zeta}_{\LpR{2+\delta}{2}}  \right) .
\end{align}
Finally, notice that:
\begin{align*}
\abs{\zeta({\bf0})} &= \abs{\phi({\bf0})	F({\bf0})} = \norm{F}_{\LpR{\infty}{2}},
\end{align*}
and that
\begin{align*}
\norm{\nabla\zeta}_{\LpR{2+\delta}{2}} &= \norm{\nabla(\phi F)}_{\LpR{2+\delta}{2}}\\
&\leq \norm{\nabla\phi}_{\LpR{\infty}{2}}\norm{F}_{\LpR{2+\delta}{2}} + \norm{\nabla F}_{\LpR{2+\delta}{2}}\norm{\phi}_{\LpR{\infty}{2}}\\
&\leq C\left( \frac{\norm{F}_{\LpR{2+\delta}{2}}}{R}+ \norm{\nabla F}_{\LpR{2+\delta}{2}}\right).
\end{align*}
This proves inequality \eqref{logarithmic_inequality_3D_HM_stated}.
\end{proof}

\begin{corollary}\label{periodic_logarithmic_inequality}
Let $F: \nT^2 \maps \nR$ be a scalar valued periodic function with periodic domain $\nT^2$, $F \in W^{1,2+\delta}(\nT^2)$, and $\delta>0$ be given. Then for any $\lambda \in (0,\infty)$ the following logarithmic inequality holds:
\begin{align}\label{logarithmic_inequality_3D_HM_P}
&\norm{F}_{L^{\infty}(\nT^2)} \leq \notag \\
&\qquad C_{\delta,\lambda}\max\set{1,\sup_{q\geq2} \frac{\norm{F}_{L^q(\nT^2)}}{q^{\lambda} L^{2/q}}}\log^{\lambda}\left(e^2 + C L^{\delta/(2+\delta)}\norm{F}_{W^{1,2+\delta}(\nT^2)}\right).
\end{align}
where $C_{\delta,\lambda}$ is an absolute positive constant that depends on $\delta$ and $\lambda$.
\end{corollary}
\begin{proof}
Since $F$ is a periodic function with periodic domain $\nT^2 = [0,L]^2$, in the above proof we may choose $R=4L$. Notice that for any $p>1$
$$\norm{\zeta}_{L^p(B_{4L})}\leq \norm{\phi}_{L^\infty(B_{4L})}\norm{F}_{L^p(B_{4L})} \leq 4^{1/p}\norm{F}_{L^p(\nT^2)},$$
and that

\begin{align*}
\norm{\nabla\zeta}_{L^p(B_{4L})} &\leq \norm{\phi}_{L^\infty(B_{4L})}\norm{\nabla F}_{L^p(B_{4L})} + \norm{\nabla \phi}_{L^\infty(B_{4L})}\norm{F}_{L^p(B_{4L})}\\
&\leq C4^{1/p}\left(\frac{\norm{F}_{L^p(\nT^2)}}{L}+\norm{\nabla F}_{L^p(\nT^2)}\right) = C4^{1/p}\norm{F}_{W^{1,p}(\nT^2)}.
\end{align*}
Thus, one can follow the steps of the proof of the previous lemma to prove inequality \eqref{logarithmic_inequality_3D_HM_P}.
\end{proof}

\begin{proposition}\label{nabla_h_u_Linfty_estimate_simplified_3D_HM}
Let $\omega \in \LpP{\infty}\cap\LzWhP{\infty}{1}{4}$, $\omega \not= 0$. Let $\bu = K\convh \omega$. Then $\nabla_h \bu \in \LpP{\infty}$ and,
\begin{align}
\norm{\nabla_h \bu}_{\LpP{\infty}} & \leq C \norm{\omega}_{\LpP{\infty}}\log\left(e^2+ CL^2\frac{\norm{\nabla_h\omega}_{\LzLphP{\infty}{4}}}{\norm{\omega}_{\LpP{2}}}\right).
\end{align}
\end{proposition}
\begin{proof}
Due to \eqref{elliptic_regularity_Yudovich} we have,
\begin{align*}
\norm{\nabla_h \bu(z)}_{\LphP{q}} \leq C q \norm{\omega(z)}_{\LphP{q}} \leq &C qL^{2/q} \norm{\omega(z)}_{\LphP{\infty}},
\end{align*}
and
\begin{align*}
\norm{\nabla_h \bu(z)}_{\WhP{1}{4}} \leq &C \norm{\nabla_h\omega(z)}_{\LphP{4}},
\end{align*}
for every $q \in (1,\infty)$ and every $z\in[0,L]$, where the constant $C$ is independent of $z$.
Set $$F =  L ^{3/2}\frac{\nabla_h\bu}{\norm{\omega}_{\LpP{2}}},$$
and apply the logarithmic inequality \eqref{logarithmic_inequality_3D_HM_P} with $\lambda=1$ and $\delta=2$, we obtain
\begin{align*}
\norm{\nabla_h\bu(z)}_{\LphP{\infty}} &\leq C\max\set{ \frac{\norm{\omega}_{\LpP{2}}}{L^{3/2}}, \norm{\omega(z)}_{\LphP{\infty}}} \log\left(e^2 + CL^2 \frac{\norm{\nabla_h\omega(z)}_{\LphP{4}}}{\norm{\omega}_{\LpP{2}}}\right).
\end{align*}
for every $z\in[0,L]$. This implies that
\begin{align*}
\norm{\nabla_h \bu}_{\LpP{\infty}} & \leq C \norm{\omega}_{\LpP{\infty}}\log\left(e^2+ CL^2\frac{\norm{\nabla_h\omega}_{\LzLphP{\infty}{4}}}{\norm{\omega}_{\LpP{2}}}\right).
\end{align*}
\end{proof}

\begin{corollary}
Let $\theta, \eta \in \LpP{\infty}\cap\LzWhP{\infty}{1}{4}$, $\nabla_h\times \bu = \frac{1}{2L}(\theta-\eta)$, $(\theta,\eta)\not = {\bf 0}$. Then $\nabla_h \bu \in \LpP{\infty}$, and
\begin{align}\label{logarithmic_inequality_simplified_3D_HM}
\norm{\nabla_h \bu}_{\LpP{\infty}} & \leq \frac{C}{L}\norm{(\theta,\eta)}_{\nLpP{\infty}}\log\left(e^2+ CL^2\frac{\norm{(\nabla_h\theta,\nabla_h\eta)}_{\nLzLphP{\infty}{4}}}{\norm{(\theta,\eta)}_{\nLpP{2}}}\right).
\end{align}
\end{corollary}
\begin{proof}
In this case one can choose,
$$ F = L^{5/2} \frac{\nabla_h\bu}{\norm{(\theta,\eta)}_{\nLpP{2}}}$$
in \eqref{logarithmic_inequality_3D_HM_stated} and proceed as in the above proof of Proposition \ref{nabla_h_u_Linfty_estimate_simplified_3D_HM}.
\end{proof}

\begin{proposition}\label{nonlinearity_L2_estimate_simplified_3D_HM}
Let $\omega \in \LpP{\infty}$, and $\bu = K\convh \omega$. Suppose that $\psi \in \LzWhP{\infty}{1}{4}$, then $(u\cdot\nabla_h)\psi \in \LpP{2}$ and
\begin{align}
\label{L2_estimate_nonlinearity_simplified_HM}
\norm{(\bu\cdot\nabla_h)\psi}_{\LpP{2}} &\leq C L^{2}\norm{\omega}_{\LpP{\infty}}\norm{\nabla_h\psi}_{\LzLphP{\infty}{4}}.
\end{align}
When $\omega = \frac{1}{2L}(\theta-\eta)$ with $\theta, \eta \in \LpP{\infty}$, then
\begin{align}
\label{L2_estimate_nonlinearity_simplified_HM_theta_eta}
\norm{(\bu\cdot\nabla_h)\psi}_{\LpP{2}} &\leq C L\norm{(\theta,\eta)}_{\nLpP{\infty}}\norm{\nabla_h\psi}_{\LzLphP{\infty}{4}}.
\end{align}
\end{proposition}
\begin{proof}
Since $\int_{\nT^2}\bu(x,y,z)\,dx\,dy = 0$, by the Sobolev type estimate \eqref{L4_to_H1}, the elliptic regularity estimate \eqref{elliptic_regularity_Yudovich}, and the Poincar\'e inequality \eqref{Poincare_inequality} we have
$$\norm{\bu(z)}_{\LphP{4}}^2 \leq C\norm{\bu(z)}_{\LphP{2}}\norm{\bu(z)}_{H^1_h(\nT^3)}\leq CL\norm{\bu(z)}_{H^1_h(\nT^3)}^2\leq CL\norm{\omega(z)}_{\LphP{2}}^2,$$
for a.e. $z\in[0,L]$. Then, by H\"older's inequality and the above inequality, we obtain
\begin{align*}
\norm{(\bu\cdot\nabla_h)\psi(z)}_{\LphP{2}}^2 &\leq \norm{\bu(z)}_{\LphP{4}}^2\norm{\nabla_h\psi(z)}_{\LphP{4}}^2, \notag \\
& \leq CL\norm{\omega(z)}_{\LphP{2}}^2 \norm{\nabla_h\psi(z)}_{\LphP{4}}^2.
\end{align*}
Integrating the last inequality with respect to $z$ over $[0,L]$ will imply:
\begin{align} \label{L2h_nonlinearity_estimate_simplified_3D_HM}
\norm{(\bu\cdot\nabla_h)\psi}_{\LpP{2}} &\leq C L^{1/2}\norm{\omega}_{\LpP{2}}\norm{\nabla_h\psi}_{\LzLphP{\infty}{4}} \notag\\
&\leq C L^2\norm{\omega}_{\LpP{\infty}}\norm{\nabla_h\psi}_{\LzLphP{\infty}{4}}.
\end{align}
When, $\omega = \frac{1}{2L}(\theta-\eta)$, its enough to notice that $\norm{\omega}_{\LpP{\infty}}\leq \frac{1}{L}\norm{(\theta,\eta)}_{\nLpP{\infty}}$. This concludes the proof.
\end{proof}

\begin{corollary}
\label{nonlinearity_estimate_simplified_3D_HM}
Let $\omega \in \LpP{2}$, and $\bu = K\convh \omega$. Suppose that $\psi \in \LzWhP{\infty}{1}{4}$ and $\phi \in \LpP{2}$, then
\begin{align}
\left| \left<(\bu\cdot \nabla_h) \psi, \phi \right>\right| \leq C L^{1/2}\norm{\omega}_{\LpP{2}}\norm{\nabla_h\psi}_{\LzLphP{\infty}{4}} \norm{\phi}_{\LpP{2}}
\end{align}
\end{corollary}
\begin{proof}
By H\"oldar's inequality and \eqref{L2h_nonlinearity_estimate_simplified_3D_HM} in the proof of Proposition \ref{nonlinearity_L2_estimate_simplified_3D_HM} we get
\begin{align*}
\left| \left<(\bu\cdot \nabla_h) \psi, \phi \right>\right| &\leq \norm{(\bu\cdot \nabla_h) \psi}_{\LpP{2}} \norm{\phi}_{\LpP{2}}\\
&\leq C L^{1/2}\norm{\omega}_{\LpP{2}}\norm{\nabla_h\psi}_{\LzLphP{\infty}{4}} \norm{\phi}_{\LpP{2}}.
\end{align*}
\end{proof}
\bigskip
\section{The ``Linearized" System}

Our task is to prove the global existence and uniqueness of solution of system \eqref{simplified_3d_HM_theta_eta}. We will prove our result using an iteration procedure. For this purpose, we introduce and study, in this section, a ``linearized" version of \eqref{simplified_3d_HM_theta_eta}, about a given state ${\tilde \omega}$, which will play a crucial role in the iteration procedure. The ``linearized" system is considered subject to periodic boundary conditions with a basic domain $\nT^3 = [0,L]^3 \subset \nR^3$, over any fixed arbitrary time interval $[0,T]$ and is given by
\begin{subequations}
\label{linear_simplified_3D_HM}
\begin{align}
\pd{\theta}{t} + ({\tilde{\bu}}\cdot\nabla_h)\theta - U_0\pd{\theta}{z} =0, \qquad \pd{\eta}{t} + ({\tilde{\bu}}\cdot\nabla_h)\eta +U_0\pd{\eta}{z} =0,\label{linear_simplified_3D_HM_evol}\\
 {\tilde \bu} = K\convh {\tilde{\omega}},\\
\theta(0;\bx) = \theta^0(\bx),\, \eta(0;\bx) = \eta^0(\bx),
\end{align}
\end{subequations}
where ${\tilde \omega}$ is a given periodic function.
\begin{theorem}
\label{global_existence_uniqueness_linear_simplified_HM}
Let $T>0$ be given, and let ${\tilde \omega} \in C^1([0,T]; C^2_{per}(\nT^3))$, $(\theta^0, \eta^0) \in C^2_{per}(\nT^3) \times C^2_{per}(\nT^3)$. Then \eqref{linear_simplified_3D_HM} has a unique solution $(\theta, \eta) \in C^2([0,T]; C^2_{per}(\nT^3) \times C^2_{per}(\nT^3))$ with
\begin{align} \label{Linfty_estimate_linear_simplified_HM_theta_eta}
\norm{(\theta(t),\eta(t))}_{\nLpP{p}} = \norm{(\theta^0,\eta^0)}_{\nLpP{p}},
\end{align}
for all $t\in[0,T]$ and all $p\in[1,\infty]$. Moreover,
\begin{align}\label{W14_estimate_linear_simplified_HM_theta_eta}
\sup_{t\in[0,T]}&\norm{(\nabla_h\theta(t),\nabla_h\eta(t))}_{\nLzLphP{\infty}{4}}\leq e^{{\tilde J}(T)T}\norm{(\nabla_h\theta^0,\nabla_h\eta^0)}_{\nLzLphP{\infty}{4}},
\end{align}
\begin{align}\label{L2_estimate_dztheta_dzeta_linear_simplified_HM}
\sup_{t\in[0,T]}\norm{(\partial_z\theta(t),\partial_z\eta(t))}_{\nLpP{2}}&\leq {\tilde{H}(T)} T+ \norm{(\partial_z\theta^0,\partial_z\eta^0)}_{\nLpP{2}},
\end{align}
and
\begin{align}\label{L2_estimate_dttheta_dteta_linear_simplified_HM}
&\sup_{t\in[0,T]}\norm{(\partial_t\theta(t),\partial_t\eta(t))}_{\nLpP{2}}\leq {\tilde{H}(T)} T+ \norm{(\partial_z\theta^0,\partial_z\eta^0)}_{\nLpP{2}}\notag \\
&\qquad \qquad \qquad \qquad+ C L^{2} e^{{\tilde{J}(T)}T}\norm{(\nabla_h\theta^0,\nabla_h\eta^0)}_{\nLzLphP{\infty}{4}} \sup_{t\in[0,T]}\norm{{\tilde\omega(t)}}_{\LpP{\infty}}.
 \end{align}
Furthermore,
\begin{align}\label{Linfty_estimate_u_linear_simplified_HM}
&\sup_{t\in[0,T]}\norm{\nabla_h \bu(t)}_{\LpP{\infty}} \notag \\
&\leq \frac{C}{L}\norm{(\theta^0,\eta^0)}_{\nLpP{\infty}}\log \left(e^2+CL^2e^{{\tilde J}(T)T}\frac{\norm{(\nabla_h\theta^0,\nabla_h\eta^0)}_{\nLzLphP{\infty}{4}}}{\norm{(\theta^0,\eta^0)}_{\LpP{2}}}\right).
\end{align}
Here $\bu = \frac{1}{2L}K\convh(\theta-\eta)$, ${\tilde J}(T)$ and ${\tilde H}(T)$ are constants that depend on the domain $\nT^3$, the norms of ${\tilde \omega}$ and ${\tilde{\theta}}$ as well as the time $T$ and are specified in \eqref{C_tilde_simplified_3D_HM} and \eqref{K_tilde_simplified_3D_HM}, respectively, while $C$ is an absolute dimensionless positive constant.
\end{theorem}

\begin{proof}
If $(\theta^0,\eta^0)={\bf 0}$, then the unique solution of \eqref{linear_simplified_3D_HM} is $(\theta(t),\eta(t)) = {\bf 0}$, and there is nothing to prove. Therefore, we assume that $(\theta^0,\eta^0)\not={\bf 0}$. Since $ {\tilde{\omega}}\in C^1([0,T];C^2_{per}(\nT^3))$ then one can easily show by classical elliptic theory that ${\tilde{\bu}}$ $\in$ $C^1([0,T];C^2_{per}(\nT^3))$. By the method of characteristics (\cite{Evans_1998}, Chapter 3, \cite{John_1986}, Chapter 1) we can ensure the existence and the uniqueness of a classical local solution $(\theta, \eta) \in C^2([0,T_{max}];C^2_{per}(\nT^3)$ $\times C^2_{per}(\nT^3))$ of \eqref{linear_simplified_3D_HM} for some short time $0<T_{max}\leq T$. Since $({\tilde{\bu}}, -1)$ is a divergence free vector field in $\nR^3$, multiplying the evolution equations in \eqref{linear_simplified_3D_HM_evol} by $\abs{\theta}^{p-2}\theta$ and $\abs{\eta}^{p-2}\eta$, respectively, and integrate over $\nT^3$ yield that
\begin{align}\label{L2_estimate_linear_simplified_HM_theta_eta_max}
\norm{\theta(t)}_{\LpP{p}} = \norm{\theta^0}_{\LpP{p}} \quad \text{and} \quad \norm{\eta(t)}_{\LpP{p}} = \norm{\eta^0}_{\LpP{p}},
\end{align}
for any $p\in[1,\infty)$ and every $t\in[0,T_{max}]$. Thus,
\begin{subequations}\label{Linfty_estimate_linear_simplified_HM_theta_eta_max}
\begin{align}
\norm{\theta(t)}_{\LpP{\infty}} = \lim_{p\rightarrow\infty}\norm{\theta}_{\LpP{p}} = \lim_{p\rightarrow\infty}\norm{\theta^0}_{\LpP{p}} &= \norm{\theta^0}_{\LpP{\infty}}, \\
\norm{\eta(t)}_{\LpP{\infty}} = \lim_{p\rightarrow\infty}\norm{\eta}_{\LpP{p}} = \lim_{p\rightarrow\infty}\norm{\eta^0}_{\LpP{p}} &= \norm{\eta^0}_{\LpP{\infty}},
\end{align}
\end{subequations}
for every $t\in[0,T_{max}]$. The uniform estimates \eqref{L2_estimate_linear_simplified_HM_theta_eta_max} and \eqref{Linfty_estimate_linear_simplified_HM_theta_eta_max} allow us to deduce that the solution $(\theta,\eta)$ can be extended on each interval $[nT_{max}, (n+1)T_{max}]\cap[0,T]$. So, for every arbitrary $T>0$, we have found a unique classical solution $(\theta,\eta) \in C^2([0,T];C^2_{per}(\nT^3)$ $\times C^2_{per}(\nT^3))$ of \eqref{linear_simplified_3D_HM} that satisfies the estimate \eqref{Linfty_estimate_linear_simplified_HM_theta_eta}. 
Since the solution $(\theta(t),\eta(t))\in C^2_{per}(\nT^3) \times C^2_{per}(\nT^3)$, we can take $\nabla_h$ of \eqref{linear_simplified_3D_HM_evol} and multiply by $|\nabla_h\theta|^2\nabla_h\theta$ and by $|\nabla_h\eta|^2\nabla_h\eta$, respectively, and then integrate horizontally to obtain:
\begin{align*}
&\frac{1}{4} \pd{{\norm{\nabla_h\theta}_{\LphP{4}}^4}}{t} - \frac{U_0}{4} \pd{{\norm{\nabla_h\theta}_{\LphP{4}}^4}}{z} \\
&\qquad \qquad + \int_0^L\int_0^L\left((\partial_x{\tilde{\bu}}\cdot\nabla_h)\theta\cdot \partial_x\theta |\nabla_h\theta|^2+ (\partial_y{\tilde{\bu}}\cdot\nabla_h)\theta\cdot \partial_y\theta |\nabla_h\theta|^2\right) \,dx\, dy = 0,
\end{align*}
which yields:
\begin{align*}
\frac{1}{4} \pd{{\norm{\nabla_h\theta(t;z)}_{\LphP{4}}^4}}{t} - \frac{U_0}{4} \pd{{\norm{\nabla_h\theta(t;z)}_{\LphP{4}}^4}}{z} \leq \norm{\nabla_h {\tilde{\bu}}(t;z)}_{\LphP{\infty}} \norm{\nabla_h\theta(t;z)}_{\LphP{4}}^4.
\end{align*}
Notice that since the solution here is classical, all these estimates are rigorous and there is no need to justify them further. Using Proposition \ref{nabla_h_u_Linfty_estimate_simplified_3D_HM} together with the above inequality we obtain:
\begin{align}\label{simplifed_3d_HM_nabla_h_theta_pd}
\pd{{\norm{\nabla_h\theta(t;z)}_{\LphP{4}}^4}}{t} - U_0\pd{{\norm{\nabla_h\theta(t;z)}_{\LphP{4}}^4}}{z} \leq 4{\tilde J}(T) \norm{\nabla_h\theta(t;z)}_{\Lph{4}}^4,
\end{align}
where
\begin{align}\label{C_tilde_simplified_3D_HM}
{\tilde J}(T):=  C\sup_{t \in [0,T]} \norm{{\tilde \omega(t)}}_{\LpP{\infty}}\log\left(e^2+ CL^2\sup_{t \in [0,T]} \frac{\norm{\nabla_h{\tilde \omega(t)}}_{\LzLphP{\infty}{4}}}{\norm{{\tilde \omega(t)}}_{\LpP{2}}}\right).
\end{align}
Applying the method of characteristics and then Gronwall's Lemma to \eqref{simplifed_3d_HM_nabla_h_theta_pd}, one can show that,
\begin{align}
\label{W14_estimate_linear_simplified_HM}
\norm{\nabla_h\theta(t;z)}_{\LphP{4}}^4 &\leq e^{4{\tilde J}(T)t}\norm{\nabla_h\theta^0(z+U_0t)}_{\LphP{4}}^4 \notag \\
&\leq e^{4{\tilde J}(T)T}\norm{\nabla_h\theta^0}_{\LzLphP{\infty}{4}}^4,
\end{align}
for every $t\in[0,T]$, and every $z\in[0,L]$. Similar inequality holds also for $\eta$. This proves estimate \eqref{W14_estimate_linear_simplified_HM_theta_eta}. Since $\theta\in C^2([0,T];C^2_{per}(\nT^3))$, we can differentiate \eqref{linear_simplified_3D_HM_evol} with respect to $z$ and get that $\partial_z\theta$ satisfies
\begin{align}
\pd{(\partial_z \theta)}{t} + (\partial_z {\tilde{\bu}} \cdot \nabla_h)\theta + {\tilde{\bu}} \cdot \nabla_h \partial_z \theta - U_0\npd{\theta}{z}{2} = 0;
\end{align}
and a similar equation holds also for $\partial_z\eta$. Since $({\tilde{\bu}}, -1)$ is a divergence free vector in $\nR^3$, taking the $\LpP{2}$ inner product of the above equation with $\partial_z \theta$ implies that,
\begin{align*}
\frac{1}{2} \od{\norm{\partial_z\theta}_{\LpP{2}}^2}{t} + \left<(\partial_z {\tilde{\bu}}\cdot \nabla_h) \theta, \partial_z \theta \right> = 0.
\end{align*}
Therefore, by Corollary \ref{nonlinearity_estimate_simplified_3D_HM}, we have
\begin{align*}
&\frac{1}{2} \od{\norm{\partial_z\theta(t)}_{\LpP{2}}^2}{t} \leq C L^{1/2} \norm{\nabla_h\theta(t)}_{\LzLphP{\infty}{4}} \norm{\partial_z {\tilde\omega}(t)}_{\LpP{2}}\norm{\partial_z \theta(t)}_{\LpP{2}}\\
&\qquad \qquad \leq  C L^{1/2}e^{{\tilde J}(T)T} \norm{\nabla_h\theta^0}_{\LzLphP{\infty}{4}} \left(\sup_{t\in[0,T]}\norm{\partial_z {\tilde\omega}(t)}_{\LpP{2}}\right)\norm{\partial_z \theta(t)}_{\LpP{2}},
\end{align*}
where in the last inequality we used \eqref{W14_estimate_linear_simplified_HM}. Similar inequality holds also for $\partial_z\eta$.  Set
\begin{align}\label{K_tilde_simplified_3D_HM}
{\tilde H}(T) := 4C L^{1/2} e^{{\tilde J}(T)T}\norm{(\nabla_h\theta^0,\nabla_h\eta^0)}_{\nLzLphP{\infty}{4}}\sup_{t\in[0,T]}\norm{\partial_z {\tilde\omega}(t)}_{\LpP{2}}.
\end{align}
Thus,
\begin{align}
\od{\norm{(\partial_z\theta,\partial_z\eta)}_{\nLpP{2}}}{t} \leq {\tilde H}(T),
\end{align}
which implies
\begin{align}
\norm{(\partial_z\theta(t),\partial_z\eta(t))}_{\LpP{2}}\leq \norm{(\partial_z\theta^0,\partial_z\eta^0)}_{\LpP{2}} + {\tilde H}(T)T,
\end{align}
for every $t\in[0,T]$. This proves estimate \eqref{L2_estimate_dztheta_dzeta_linear_simplified_HM}. Finally, by Proposition \ref{nonlinearity_L2_estimate_simplified_3D_HM} and \eqref{W14_estimate_linear_simplified_HM} we get
\begin{align}\label{L2_estimate_nonlinearity_linear_simplified_HM}
\norm{({\tilde\bu}\cdot\nabla_h)\theta(t)}_{\LpP{2}} &\leq C L^{2}e^{{\tilde J}(T)T}\norm{\nabla_h\theta^0}_{\LzLphP{\infty}{4}}\sup_{t\in[0,T]}\norm{{\tilde \omega}(t)}_{\LpP{\infty}},
\end{align}
for every $t\in[0,T]$. Similar inequality holds also for $\norm{(\bu\cdot\nabla_h)\eta}_{\LpP{2}}$.

Next, we observe that since \eqref{linear_simplified_3D_HM_evol} is equivalent to:
\begin{align}
\pd{\theta}{t} = -({\tilde{\bu}}\cdot\nabla_h)\theta +U_0 \pd{\theta}{z},\qquad \pd{\eta}{t} = -({\tilde{\bu}}\cdot\nabla_h)\eta - U_0\pd{\eta}{z},
\end{align}
then by \eqref{L2_estimate_dztheta_dzeta_linear_simplified_HM} and \eqref{L2_estimate_nonlinearity_linear_simplified_HM}, we obtain estimate \eqref{L2_estimate_dttheta_dteta_linear_simplified_HM}.

Set $\bu = K\convh\omega$, where $\omega = \frac{1}{2L}(\theta-\eta)$. Hence, the logarithmic inequality \eqref{logarithmic_inequality_simplified_3D_HM} implies:
\begin{align*}
\norm{\nabla_h \bu(t)}_{\LpP{\infty}}&\leq \frac{C}{L} \norm{(\theta(t),\eta(t))}_{\nLpP{\infty}}\log\left(e^2+ CL^2\frac{\norm{(\nabla_h\theta(t),\nabla_h\eta(t))}_{\nLzLphP{\infty}{4}}}{\norm{(\theta(t),\eta(t))}_{\nLpP2}}\right),
\end{align*}
for every $t\in[0,T]$. Thanks to \eqref{Linfty_estimate_linear_simplified_HM_theta_eta}, $\norm{(\theta(t),\eta(t))}_{\nLpP{2}} = \norm{(\theta^0,\eta^0)}_{\nLpP{2}}$ and $\norm{(\theta(t),\eta(t))}_{\nLpP{\infty}} =\norm{(\theta^0,\eta^0)}_{\nLpP{\infty}}$, for every $t \in [0,T]$, as a result \eqref{Linfty_estimate_linear_simplified_HM_theta_eta} and \eqref{W14_estimate_linear_simplified_HM_theta_eta} imply \eqref{Linfty_estimate_u_linear_simplified_HM}; and this completes the proof.
\end{proof}
 \bigskip
\section{Global Existence and Uniqueness}
In this section, we aim to prove global existence and uniqueness of solutions to system \eqref{simplified_3d_HM_theta_eta}, subject to periodic boundary conditions with a basic domain $\nT^3 = [0,L]^3 \subset \nR^3$, over any fixed arbitrary time interval $[0,T]$. We recall in more details system \eqref{simplified_3d_HM_theta_eta}
\begin{subequations}
\label{simplified_3D_HM}
\begin{align}
\pd{\theta}{t} + (\bu\cdot\nabla_h)\theta - U_0\pd{\theta}{z} =0, \label{simplified_3D_HM_evol}\quad \pd{\eta}{t} +(\bu\cdot\nabla_h)\eta +U_0\pd{\eta}{z} =0, \\
\bu = K\convh\omega, \quad \omega = \frac{1}{2L}(\theta-\eta),\\
\theta(0;\bx) = \theta^0(\bx),\quad \eta(0;\bx) = \eta^0(\bx).
\end{align}
\end{subequations}

Next, we give a definition of a solution to system \eqref{simplified_3D_HM} or equivalently \eqref{simplified_3d_HM_theta_eta}.
\begin{definition}
\label{weak_solution_definition_simplified_3D_HM}
Let $T>0$, and $(\theta^0,\eta^0) \in \nLpP{\infty}$ with $(\nabla_h\theta^0,\nabla_h\eta^0)\in \nLzLphP{\infty}{4}$, $(\partial_z\theta^0,\partial_z\eta^0)\in \nLpP{2}$ be given. We say that $(\theta,\eta)$ is a weak solution of $\eqref{simplified_3D_HM}$, if
$$(\theta,\eta) \in C([0,T];\nLpP{2}), \quad (\nabla_h\theta,\nabla_h\eta) \in L^\infty([0,T]; \nLzLphP{\infty}{4}),$$
$$(\partial_z\theta,\partial_z\eta) \in L^\infty([0,T]; \nLpP{2}), \quad  (\partial_t \theta,\partial_t\eta) \in L^\infty([0,T]; \nLpP{2}),$$
$$\theta(0;\bx) = \theta^0(\bx) \quad \text{and} \quad \eta(0;\bx) = \eta^0(\bx),$$
such that $(\theta,\eta)$ satisfy \eqref{simplified_3D_HM} in the weak formulation (i.e., a solution in the distribution sense). That is, for any scalar test functions $\vphi(t;\bx), \psi(t;\bx)  \in C^\infty_{per}([0,T]; \nT^3)$, such that $\vphi(T;\bx) =0$, $\psi(T;\bx)=0$,
\begin{align*}
-\int_0^T\left<\theta(s),\pd{\vphi(s)}{s}\right>\,ds - \int_0^T\left<\bu\theta,\nabla_h\vphi\right>\,ds+ U_0\int_0^T\left<\theta(s), \pd{\vphi(s)}{z}\right>\,ds
&=\left<\theta^0(x),\vphi(\bx,0)\right>,\\
-\int_0^T\left<\eta(s),\pd{\psi(s)}{s}\right>\,ds - \int_0^T\left<\bu\eta,\nabla_h\psi\right>\,ds-U_0 \int_0^T\left<\eta(s), \pd{\psi(s)}{z}\right>\,ds
&=\left<\eta^0(x),\psi(\bx,0)\right>.
\end{align*}
where $\bu = K\convh \omega = \frac{1}{2L}K\convh(\theta-\eta)$.
\end{definition}
In the case of periodic boundary conditions, it will be sufficient to consider only test functions of the form
\begin{subequations}\label{test_fcns}
\begin{align}
\vphi(t;\bx) = \chi_{\bm}(t) e^{2\pi i\bm\cdot \bx} ,  \qquad \psi(t;\bx) = \gamma_{\bk}(t) e^{2\pi i\bk\cdot\bx},
\end{align}
\begin{align}
\chi_{\bm}, \gamma_{\bk} \in C^\infty([0,T]), \qquad \chi_{\bm}(T)=\gamma_{\bk}(T)=0,
\end{align}
\end{subequations}
for $\bm, \bk\in \nZ^2$, since such functions form a basis for the corresponding larger spaces of test functions.
Thus, $(\theta,\eta)$ is a solution in the distribution sense of \eqref{simplified_3D_HM} if
\begin{align}
&-\int_0^T\left<\theta(s),e^{2\pi i\bm\cdot \bx}\right>\chi^{'}_{\bm}(s)\,ds - \int_0^T\left<\bu\theta,\nabla_h e^{2\pi i\bm\cdot \bx}\right>\chi_{\bm}(s)\,ds + \notag\\
&\qquad \qquad \qquad U_0 \int_0^T\left<\theta(s), \pd{e^{2\pi i\bm\cdot \bx}}{z}\right>\chi_{\bm}(s)\,ds =\left<\theta^0(x),e^{2\pi i\bm\cdot \bx}\right>\chi_{\bm}(0), \\
&-\int_0^T\left<\eta(s),e^{2\pi i\bk\cdot \bx}\right>\gamma^{'}_{\bk}(s)\,ds - \int_0^T\left<\bu\eta,\nabla_h e^{2\pi i\bk\cdot \bx}\right>\gamma_{\bk}(s)\,ds - \notag\\
&\qquad \qquad \qquad U_0 \int_0^T\left<\eta(s), \pd{e^{2\pi i\bk\cdot \bx}}{z}\right>\gamma_{\bk}(s)\,ds =\left<\eta^0(x),e^{2\pi i\bk\cdot \bx}\right>\gamma_{\bk}(0),
\end{align}
for all $\bm, \bk \in\nZ^2$.
We notice that if $(\theta,\eta)$ is a weak solution, then by Proposition \ref{nonlinearity_L2_estimate_simplified_3D_HM}, $(\bu\cdot\nabla_h)\theta, (\bu\cdot\nabla_h)\eta \in L^\infty([0,T],\LpP{2})$. Consequently, we can apply Lemma 1.1 (\cite{Temam_2001_Th_Num}, page 169) with $X = \LpP{2}$ to conclude that:
\begin{align}
\theta &= \xi_1 + \int_0^t \left(\nabla_h\cdot(\bu\theta)(s) + U_0\pd{\theta}{z}(s)\right)\, ds, \\
\eta &= \xi_2 + \int_0^t \left(\nabla_h\cdot(\bu\eta)(s) - U_0\pd{\eta}{z}(s)\right)\, ds,
\end{align}
for some $\xi_1, \xi_2 \in \LpP{2}$, a.e. $t \in [0,T]$, and,
\begin{align}
\pd{\theta}{t}(t) &= -\nabla_h\cdot(\bu\theta)(t)+ U_0\pd{\theta}{z}(t), \\
\pd{\eta}{t}(t) &= -\nabla_h\cdot(\bu\eta)(t) - U_0\pd{\eta}{z}(t),
\end{align}
in $\LpP{2}$, a.e. $t\in[0,T]$. If $(\theta,\eta)$ is a solution of \eqref{simplified_3D_HM}, and equivalently \eqref{simplified_3d_HM_theta_eta}, according to Definition \ref{weak_solution_definition_simplified_3D_HM}, then $(\partial_z\theta,\partial_z\eta) \in L^\infty([0,T];\nLpP{2})$, $(\partial_t\theta,\partial_t\eta) \in L^\infty([0,T];\nLpP{2})$, and by Proposition \ref{nonlinearity_L2_estimate_simplified_3D_HM}, $(\bu\cdot\nabla_h)\theta, (\bu\cdot\nabla_h)\eta \in L^\infty([0,T],\LpP{2})$, then we have
\begin{align}
\pd{\theta}{t} &= -\nabla_h\cdot(\bu\theta)+ U_0\pd{\theta}{z}, \qquad \text{in} \,\, L^\infty([0,T];\LpP{2}), \\
\pd{\eta}{t} &= -\nabla_h\cdot(\bu\eta) - U_0\pd{\eta}{z}, \qquad \text{in} \,\, L^\infty([0,T];\LpP{2}).
\end{align}
That is, for any $\psi_1,\psi_2 \in L^1([0,T];\LpP{2})$,
\begin{align*}
\int_0^T\left<\pd{\theta}{t}(t),\psi_1(t)\right>\,dt &= \int_0^T\left<-\nabla_h\cdot(\bu\theta)(t), \psi_1(t)\right>\,dt  +U_0 \int_0^T\left<\pd{\theta}{z}(t), \psi_1(t)\right>\,dt,\\
\int_0^T\left<\pd{\eta}{t}(t),\psi_2(t)\right> \,dt &= \int_0^T\left<-\nabla_h\cdot(\bu\eta)(t), \psi_2(t)\right>\,dt - U_0\int_0^T\left<\pd{\eta}{z}(t), \psi_2(t)\right>\,dt .
\end{align*}
\begin{theorem}
\label{global_existence_simplified_HM_classical}
Let $(\theta^0, \eta^0) \in C^2_{per}(\nT^3) \times C^2_{per}(\nT^3)$, and $T>0$ be given. Then there exists a unique weak solution $({{\theta}},{{\eta}})$ of \eqref{simplified_3D_HM} in the sense of Definition \ref{weak_solution_definition_simplified_3D_HM} such that:
\begin{align}\label{estimate_a_simplified_3D_HM}
\sup_{t\in[0,T]}\norm{({ \theta}(t), { \eta}(t))}_{\nLpP{p}} = \norm{(\theta^0, \eta^0)}_{\nLpP{p}},
\end{align}
for any $p\in[1,\infty]$,
\begin{align}
\sup_{t\in[0,T]}\norm{(\nabla_h{\theta}(t),\nabla_h{\eta}(t))}_{\nLzLphP{\infty}{4}} \leq {J_0(T)}\norm{(\nabla_h\theta^0,\nabla_h\eta^0)}_{\nLzLphP{\infty}{4}},
\end{align}
\begin{align}
\sup_{t\in[0,T]}\norm{(\partial_z{ \theta}(t), \partial_z{ \eta}(t))}_{\nLpP{2}} \leq {H_0(T)}\norm{(\partial_z\theta^0,\partial_z\eta^0)}_{\nLpP{2}},
\end{align}
and
\begin{align}\label{estimate_b_simplified_3D_HM}
&\sup_{t\in[0,T]}\norm{(\partial_t{ \theta}(t),\partial_t{ \eta}(t))}_{\nLpP{2}} \leq H_0(T)\norm{(\partial_z\theta^0,\partial_z\eta^0)}_{\nLpP{2}} \notag \\
&\qquad \qquad +  CLJ_0(T)\norm{(\theta^0,\eta^0)}_{\nLpP{\infty}}\norm{(\nabla_h\theta^0,\nabla_h\eta^0)}_{\nLzLphP{\infty}{4}}.
\end{align}
Furthermore,
\begin{align}
&\sup_{t\in[0,T]}\norm{\nabla_h { \bu}(t)}_{\LpP{\infty}}\leq \notag \\
& \frac{C}{L}\norm{(\theta^0,\eta^0)}_{\nLpP{\infty}}\log \left(e^2+CL^2J_0(T)\frac{\norm{(\nabla_h\theta^0,\nabla_h\eta^0)}_{\nLzLphP{\infty}{4}}}{\norm{(\theta^0,\eta^0)}_{\nLpP{2}}}\right).\qquad
\end{align}
Here ${{\bu}} = \frac{1}{2}K\convh(\theta-\eta)$. $C$ is a positive dimensionless constant, $C_0, J_0(T)$ and $H_0(T)$ are constants that depend on the domain $\nT^3$ as well as the norms of the initial data $(\theta^0,\eta^0)$ specified in \eqref{C_0_simplified_3D_HM}, \eqref{J_0_simplified_3D_HM} and \eqref{H_0_simplified_3D_HM}, respectively.
\end{theorem}
\begin{proof}
If $(\theta^0,\eta^0)={\bf 0}$, then $(\theta(t),\eta(t)) = {\bf 0}$ is a solution, and by the uniqueness part that we show later, it will be the only solution, and the statement of the theorem follows. Therefore, we assume that $(\theta^0,\eta^0)\not={\bf 0}$. The idea of the proof is involving the construction of a sequence of approximate solutions of \eqref{simplified_3D_HM} based on the ``linearized" system, as it was introduced in section 3; and then passing to the limit to a weak solution. We proceed with the following steps.

{\bf{Step 1}}: Construction of a sequence of approximate solutions.\\
We consider the sequence $\{(\theta^n,\eta^n)\}_{n=0}^{\infty}$, where, for $n = 1, 2 ,\ldots$, $(\theta^n,\eta^n)$ is the unique solution of the linear system

\begin{subequations}
\label{simplified_3D_HM_n}
\begin{align}
\pd{\theta^n}{t} + (\bu^{n-1}\cdot\nabla_h)\theta - U_0\pd{\theta^n}{z} =0, \quad \pd{\eta^n}{t} + (\bu^{n-1}\cdot\nabla_h)\eta + U_0\pd{\eta^n}{z} =0, \label{simplified_3D_HM_theta_eta_n}\\
\bu^{n-1} = K\convh \omega^ {n-1}, \, \omega^{n-1}= \frac{1}{2L} \left(\theta^{n-1}-\eta^{n-1}\right),\\
\theta^n(0;\bx) = \theta^0(\bx),\, \eta^n(0;\bx) = \eta^0(\bx).
\end{align}
\end{subequations}
Since $(\theta^0,\eta^0) \in C^2_{per}(\nT^3)\times C^2_{per}(\nT^3)$, one can use induction steps to establish the existence and the uniqueness of the sequence  of solutions $(\theta^n,\eta^n) \in C^2([0,T];C^2_{per}(\nT^3) \times C^2_{per}(\nT^3))$ to \eqref{simplified_3D_HM_n} by virtue of Theorem \ref{global_existence_uniqueness_linear_simplified_HM}. In addition, thanks to \eqref{Linfty_estimate_linear_simplified_HM_theta_eta}, one also obtains the following uniform estimate:
\begin{align}
\label{Linfty_estimate_linear_simplified_HM_n_theta_eta}
\norm{(\theta^n(t),\eta^n(t))}_{\nLpP{p}} &= \norm{(\theta^0, \eta^0)}_{\nLpP{p}},
\end{align}
for any $p\in[1,\infty]$, and all $t\in[0,T]$. In turn, this implies that for all $n = 0, 1, 2 ,\ldots$,
\begin{align}
\label{Linfty_omega_simplified_3D_HM_n}
\sup_{t\in[0,T]}\norm{\omega^n(t)}_{\LpP{\infty}} &\leq \frac{1}{L}\norm{(\theta^n(t),\eta^n(t))}_{\nLpP{\infty}} = \frac{1}{L}\norm{(\theta^0, \eta^0)}_{\nLpP{\infty}}.
\end{align}
\begin{claim}\label{claim_1_simplified_3D_HM}
For all $n = 0, 1, 2 ,\ldots$,
\begin{align}\label{W14_estimate_linear_simplified_HM_n_theta_eta}
\sup_{t\in[0,T]}\norm{(\nabla_h\theta^n(t),\nabla_h\eta^n(t))}_{\nLzLphP{\infty}{4}} \leq {J_0(T)}\norm{(\nabla_h\theta^0,\nabla_h\eta^0)}_{\nLzLphP{\infty}{4}},
\end{align}
where $J_0(T)$ is a constant that depends on $L$ and the norms of the initial data $(\theta^0,\eta^0)$, as well as the time $T$ and is specified in \eqref{J_0_simplified_3D_HM}, below.
\end{claim}
{\it{Proof of Claim 4.1:}} By a similar argument as in the proof of Theorem \ref{global_existence_uniqueness_linear_simplified_HM}, we take $\nabla_h$ of \eqref{simplified_3D_HM_theta_eta_n} and then multiply respectively by $|\nabla_h\theta^n|^2\nabla_h\theta^n$ and $|\nabla_h\eta^n|^2\nabla_h\eta^n$ and then integrate horizontally and obtain
\begin{subequations}\label{pd-W14_simplified_linear_3d_HM_theta_eta_n}
\begin{align}
\pd{\norm{\nabla_h\theta^n}_{\LphP{4}}^4}{t} -{U_0} \pd{\norm{\nabla_h\theta^n}_{\LphP{4}}^4}{z} &\leq 4\norm{\nabla_h \bu^{n-1}(z)}_{\LphP{\infty}}\norm{\nabla_h\theta^n}_{\LphP{4}}^4,\\
\pd{\norm{\nabla_h\eta^n}_{\LphP{4}}^4}{t} +{U_0}\pd{\norm{\nabla_h\eta^n}_{\LphP{4}}^4}{z} &\leq 4\norm{\nabla_h \bu^{n-1}(z)}_{\LphP{\infty}}\norm{\nabla_h\eta^n}_{\LphP{4}}^4.
\end{align}
\end{subequations}
By the logarithmic inequality \eqref{logarithmic_inequality_simplified_3D_HM} and the estimate \eqref{Linfty_estimate_linear_simplified_HM_n_theta_eta}, we have:
\begin{align}\label{nabla_h_u_simplified_3D_HM_n}
&\norm{\nabla_h \bu^{n-1}(t)}_{\LpP{\infty}} \leq \notag \\
& \quad \frac{C_1}{L}\norm{(\theta^{0},\eta^{0})}_{\nLpP{\infty}}\log\left(e^2+ C_1L^2\frac{\norm{(\nabla_h\theta^{n-1}(t),\nabla_h\eta^{n-1}(t))}_{\nLzLphP{\infty}{4}}}{\norm{(\theta^{0},\eta^{0})}_{\nLpP{2}}}\right),
\end{align}
for some positive dimensionless constant $C_1$ and all $t\in[0,T]$. This implies that
\begin{align*}
\norm{\nabla_h \bu^{0}}_{\LpP{\infty}} & \leq C_0,
\end{align*}
where
\begin{align}\label{C_0_simplified_3D_HM}
C_0:= \frac{C_1}{L}\norm{(\theta^0,\eta^0)}_{\nLpP{\infty}}\log\left(e^2+ C_1L^2\frac{\norm{(\nabla_h\theta^0,\nabla_h\eta^0)}_{\LzLphP{\infty}{4}}}{\norm{(\theta^0,\eta^0)}_{\nLpP{2}}}\right).
\end{align}
Using the above estimate, and applying the method of characteristics, followed by Gronwall's lemma, to \eqref{pd-W14_simplified_linear_3d_HM_theta_eta_n}, for $n=1$, implies:
\begin{align*}
\norm{(\nabla_h\theta^1(t),\nabla_h\eta^1(t))}_{\nLzLphP{\infty}{4}} \leq e^{C_0t}\norm{(\nabla_h\theta^0,\nabla_h\eta^0)}_{\nLzLphP{\infty}{4}},
\end{align*}
for all $t\in[0,T]$. As a result, and thanks to \eqref{nabla_h_u_simplified_3D_HM_n}, we have
\begin{align*}
&\norm{\nabla_h \bu^{1}(t)}_{\LpP{\infty}}  \leq \frac{C_1}{L}\norm{(\theta^0,\eta^0)}_{\nLpP{\infty}}\log\left(e^2+ C_1L^2e^{C_0t}\frac{\norm{(\nabla_h\theta^0,\nabla_h\eta^0)}_{\LzLphP{\infty}{4}}}{\norm{(\theta^0,\eta^0)}_{\nLpP{2}}}\right)\\
& \qquad\leq \frac{C_1}{L}\norm{(\theta^0,\eta^0)}_{\nLpP{\infty}}\log\left(e^2e^{C_0t}+ C_1L^2e^{C_0t}\frac{\norm{(\nabla_h\theta^0,\nabla_h\eta^0)}_{\LzLphP{\infty}{4}}}{\norm{(\theta^0,\eta^0)}_{\nLpP{2}}}\right)\\
&\qquad \leq \frac{C_1}{L}\norm{(\theta^0,\eta^0)}_{\nLpP{\infty}}\left(C_0t+ \log\left(e^2+ C_1L^2\frac{\norm{(\nabla_h\theta^0,\nabla_h\eta^0)}_{\LzLphP{\infty}{4}}}{\norm{(\theta^0,\eta^0)}_{\nLpP{2}}}\right)\right)
\end{align*}
\begin{align*}
&\qquad = C_0\left(1+ \frac{C_1}{L}t\norm{(\theta^0,\eta^0)}_{\nLpP{\infty}}\right),
\end{align*}
for al $t\in[0,T]$. Again, applying the method of characteristics followed by Gronwall's lemma to \eqref{pd-W14_simplified_linear_3d_HM_theta_eta_n} for $n=2$ implies
\begin{align*}
\norm{(\nabla_h\theta^2(t),\nabla_h\eta^2(t))}_{\nLzLphP{\infty}{4}} \leq e^{C_0t\left(1+ \frac{1}{2}\frac{C_1}{L}t\norm{(\theta^0,\eta^0)}_{\nLpP{\infty}}\right)}\norm{(\nabla_h\theta^0,\nabla_h\eta^0)}_{\nLzLphP{\infty}{4}},
\end{align*}
for all $t\in[0,T]$. Next, we proceed by induction. Assume that:
\begin{align*}
\norm{(\nabla_h\theta^n(t),\nabla_h\eta^n(t))}_{\nLzLphP{\infty}{4}}&\leq e^{g_n(t)t}\norm{(\nabla_h\theta^0,\nabla_h\eta^0)}_{\nLzLphP{\infty}{4}},
\end{align*}
for all $t \in [0,T]$, where $$g_n(t):=C_0\sum_{j=0}^{n-1} {\frac{\left(\frac{C_1}{L}t\norm{(\theta^0,\eta^0)}_{\nLpP{\infty}}\right)^j}{(j+1)!} }.$$
Then by \eqref{nabla_h_u_simplified_3D_HM_n} we show that:
\begin{align*}
&\norm{\nabla_h \bu^{n}(t)}_{\LpP{\infty}}  \leq \frac{C_1}{L}\norm{(\theta^{0},\eta^{0})}_{\nLpP{\infty}}\log\left(e^2+ C_1L^2\frac{\norm{(\nabla_h\theta^n(t),\nabla_h\eta^{n}(t))}_{\nLzLphP{\infty}{4}}}{\norm{(\theta^{0},\eta^{0})}_{\nLpP{2}}}\right)\\
& \leq \frac{C_1}{L}\norm{(\theta^{0},\eta^{0})}_{\nLpP{\infty}}\log\left(e^2+ C_1L^2e^{g_n(t)t}\frac{\norm{(\nabla_h\theta^0,\nabla_h\eta^0)}_{\nLzLphP{\infty}{4}}}{\norm{(\theta^0,\eta^0)}_{\nLpP{2}}}\right),
\end{align*}
for all $t\in[0,T]$. One can easily see, by the properties of the exponential and logarithmic functions, that

\begin{align*}
&C_1\log\left(e^2+ C_1L^2e^{g_n(t)t}\frac{\norm{(\nabla_h\theta^0,\nabla_h\eta^0)}_{\nLzLphP{\infty}{4}}}{\norm{(\theta^0,\eta^0)}_{\nLpP{2}}}\right)\\
&\leq C_1\left({g_n(t)t}+ \log\left(e^2+ C_1L^2\frac{\norm{(\nabla_h\theta^0,\nabla_h\eta^0)}_{\nLzLphP{\infty}{4}}}{\norm{(\theta^0,\eta^0)}_{\nLpP{2}}}\right)\right)\\
& = \frac{C_0L}{\norm{(\theta^0,\eta^0)}_{\nLpP{\infty}}} + C_1g_n(t)t\\
&=\frac{C_0L}{\norm{(\theta^0,\eta^0)}_{\nLpP{\infty}}}\sum_{j=0}^{n} \frac{\left(\frac{C_1}{L}t\norm{(\theta^0,\eta^0)}_{\nLpP{\infty}}\right)^j}{(j+1)!} = \frac{g_{n+1}(t)L}{\norm{(\theta^0,\eta^0)}_{\nLpP{\infty}}}.
\end{align*}
Thus,
\begin{align*}
\norm{\nabla_h \bu^{n}(t)}_{\LpP{\infty}} & \leq g_{n+1}(t),
\end{align*}
for all $t\in[0,T]$. Applying the method of characteristics followed by Gronwall's lemma to \eqref{pd-W14_simplified_linear_3d_HM_theta_eta_n} for $n+1$ implies:
\begin{align}
\norm{(\nabla_h\theta^{n+1}(t),\nabla_h\eta^{n+1}(t))}_{\nLzLphP{\infty}{4}}&\leq e^{g_{n+1}(t)t}\norm{(\nabla_h\theta^0,\nabla_h\eta^0)}_{\nLzLphP{\infty}{4}}.
\end{align}
for all $t\in [0,T]$. Setting
\begin{align}\label{J_0_simplified_3D_HM}
J_0(T) = e^{C_0Te^{\frac{C_1}{L}t\norm{(\theta^0,\eta^0)}_{\nLpP{\infty}}}},
\end{align}
thus we have proved, for all $n = 0, 1, 2 ,\ldots$,
\begin{align}
\norm{(\nabla_h\theta^{n}(t),\nabla_h\eta^{n}(t))}_{\nLzLphP{\infty}{4}} &\leq J_0(T) \norm{(\nabla_h\theta^0,\nabla_h\eta^0)}_{\nLzLphP{\infty}{4}},
\end{align}
for all $t\in[0,T]$. This completes the proof of the claim.
\begin{claim}\label{claim_2_simplified_3D_HM}
For all $n = 0, 1, 2 ,\ldots$,
\begin{align}
\label{L2_dztheta_dzeta_estimate_simplified_HM_n}
\sup_{t\in[0,T]}\norm{(\partial_z{ \theta^n}(t), \partial_z{\eta^n}(t))}_{\nLpP{2}} \leq {H_0(T)}\norm{(\partial_z\theta^0,\partial_z\eta^0)}_{\nLpP{2}},
\end{align}
where $H_0(T)$ is a constant that depends on $L$ and the norms of the initial data $(\theta^0,\eta^0)$, as well as the time $T$, and is specified in \eqref{H_0_simplified_3D_HM}.
\end{claim}
{\it{Proof of Claim 4.2:}}
The evolution equations of $\partial_z\theta^n$ and $\partial_z\eta^n$ are, respectively, given by:
\begin{subequations}
\begin{align}
\pd{\partial_z\theta^n}{t} + (\bu^{n-1}\cdot\nabla_h)\partial_z\theta^n +( \partial_z\bu^{n-1}\cdot\nabla_h)\theta^n -U_0\npd{\theta^n}{z}{2} = 0, \label{pd_dztheta_evol_n}\\
\pd{\partial_z\eta^n}{t} + (\bu^{n-1}\cdot\nabla_h)\partial_z\eta^n + (\partial_z\bu^{n-1}\cdot\nabla_h)\eta^n + U_0\npd{\eta^n}{z}{2} = 0.\label{pd_dzeta_evol_n}
\end{align}
\end{subequations}
When we take the $\LpP{2}$ inner product of \eqref{pd_dztheta_evol_n} with $\partial_z\theta^n$, and of \eqref{pd_dzeta_evol_n} with $\partial_z\eta^n$, we obtain:
\begin{align*}
\frac{1}{2}\od{\norm{\partial_z\theta^n}_{\LpP{2}}^2}{t} + \left<(\partial_z\bu^{n-1}\cdot\nabla_h)\theta^n, \partial_z\theta^n\right> = 0,\\
\frac{1}{2}\od{\norm{\partial_z\eta^n}_{\LpP{2}}^2}{t} + \left<(\partial_z\bu^{n-1}\cdot\nabla_h)\eta^n, \partial_z\eta^n\right> = 0.
\end{align*}
Since $\norm{\partial_z\omega^{n-1}}_{\LpP{2}} \leq \frac{1}{L}\norm{(\partial_z\theta^{n-1},\partial_z\eta^{n-1})}_{\nLpP{2}}$,
then by Corollary \ref{nonlinearity_estimate_simplified_3D_HM} and Claim \ref{claim_1_simplified_3D_HM}, we have:
\begin{align}\label{pd_dztheta_dzeta_estimate_n}
&\od{\norm{(\partial_z\theta^n,\partial_z\eta^n)}_{\nLpP{2}}}{t}\leq \notag \\
& \qquad \qquad \frac{C_2}{L^{1/2}}J_0(T)\norm{(\nabla_h\theta^0,\nabla_h\eta^0)}_{\nLzLphP{\infty}{4}}\norm{(\partial_z\theta^{n-1},\partial_z\eta^{n-1})}_{\nLpP{2}},
\end{align}
for some positive dimensionless constant $C_2$. When $n=1$, we have:
\begin{align*}
\od{\norm{(\partial_z\theta^1,\partial_z\eta^1)}_{\nLpP{2}}}{t}\leq \frac{C_2}{L^{1/2}}J_0(T)\norm{(\nabla_h\theta^0,\nabla_h\eta^0)}_{\nLzLphP{\infty}{4}}\norm{(\partial_z\theta^{0},\partial_z\eta^{0})}_{\nLpP{2}}.
\end{align*}
After we integrate with respect to time we obtain:
\begin{align}\label{od_dztheta_1_3d_simplified_HM}
&\norm{(\partial_z\theta^1(t),\partial_z\eta^1(t))}_{\nLpP{2}} \leq \notag \\
&\qquad \quad \norm{(\partial_z\theta^{0},\partial_z\eta^{0})}_{\nLpP{2}}\left(1+ \frac{C_2}{L^{1/2}}J_0(T)t\norm{(\nabla_h\theta^0,\nabla_h\eta^0)}_{\nLzLphP{\infty}{4}}\right).
\end{align}
For $n=1$, the inequality \eqref{pd_dztheta_dzeta_estimate_n} reads:
\begin{align*}
\od{\norm{(\partial_z\theta^2,\partial_z\eta^2)}_{\nLpP{2}}}{t}\leq \frac{C_2}{L^{1/2}}J_0(T)\norm{(\nabla_h\theta^0,\nabla_h\eta^0)}_{\nLzLphP{\infty}{4}}\norm{(\partial_z\theta^{1},\partial_z\eta^{1})}_{\nLpP{2}}.
\end{align*}
Thus, using \eqref{od_dztheta_1_3d_simplified_HM} and integrating with respect to time, we reach
\begin{align*}
&\norm{(\partial_z\theta^2(t),\partial_z\eta^2(t))}_{\nLpP{2}} \leq \\
& \qquad \qquad  \norm{(\partial_z\theta^{0},\partial_z\eta^{0})}_{\nLpP{2}}\sum_{k=0}^{2}\frac{1}{k!}\left(\frac{C_2}{L^{1/2}}J_0(T)t\norm{(\nabla_h\theta^0,\nabla_h\eta^0)}_{\nLzLphP{\infty}{4}}\right)^k.
\end{align*}
We proceed by induction. Suppose that:
\begin{align}\label{od_dztheta_induction_3d_simplified_HM}
&\norm{(\partial_z\theta^n(t),\partial_z\eta^n(t))}_{\nLpP{2}}\leq \notag \\
& \qquad \norm{(\partial_z\theta^{0},\partial_z\eta^{0})}_{\nLpP{2}}\sum_{k=0}^{n}\frac{1}{k!}\left(\frac{C_2}{L^{1/2}}J_0(T)t\norm{(\nabla_h\theta^0,\nabla_h\eta^0)}_{\nLzLphP{\infty}{4}}\right)^k.
\end{align}
After we substitute \eqref{od_dztheta_induction_3d_simplified_HM} in inequality \eqref{pd_dztheta_dzeta_estimate_n}, for $n+1$, and integrate with respect to time we have,
\begin{align*}
&\norm{(\partial_z\theta^{n+1}(t),\partial_z\eta^{n+1}(t))}_{\nLpP{2}}\leq \\
& \qquad \qquad \norm{(\partial_z\theta^{0},\partial_z\eta^{0})}_{\nLpP{2}}\sum_{k=0}^{n+1}\frac{1}{k!}\left(\frac{C_2}{L^{1/2}}J_0(T)t\norm{(\nabla_h\theta^0,\nabla_h\eta^0)}_{\nLzLphP{\infty}{4}}\right)^k.
\end{align*}
Set,
\begin{align}\label{H_0_simplified_3D_HM}
H_0(T) = e^{\frac{C_2}{L^{1/2}}J_0(T)T\norm{(\nabla_h\theta^0,\nabla_h\eta^0)}_{\nLzLphP{\infty}{4}}},
\end{align}
then we have just proven that for $n = 0, 1, 2 ,\ldots$,
\begin{align}
\norm{(\partial_z\theta^{n+1}(t),\partial_z\eta^{n+1}(t))}_{\nLpP{2}}\leq H_0(T) \norm{(\partial_z\theta^{0},\partial_z\eta^{0})}_{\nLpP{2}},
\end{align}
for all $t\in[0,T]$. This concludes the proof of Claim \ref{claim_2_simplified_3D_HM}.
Finally, by Proposition \ref{nonlinearity_L2_estimate_simplified_3D_HM} and estimate \eqref{Linfty_estimate_linear_simplified_HM_n_theta_eta}, for all $n = 0, 1, 2 ,\ldots$,

\begin{align}
\label{L2_estimate_nonlinearity_linear_simplified_HM_n_theta}
&\norm{(\bu^{n-1}\cdot\nabla_h)\theta^n(t)}_{\LpP{2}}+\norm{(\bu^{n-1}\cdot\nabla_h)\eta^n(t)}_{\LpP{2}} \notag \\
&\qquad \leq CL \norm{(\theta^0,\eta^0)}_{\nLpP{\infty}}\norm{(\nabla_h\theta^n(t),\nabla_h\eta^n(t))}_{\nLzLphP{\infty}{4}}.
\end{align}
As $(\theta^n, \eta^n)$ is a classical solution then by \eqref{L2_dztheta_dzeta_estimate_simplified_HM_n}
and \eqref{L2_estimate_nonlinearity_linear_simplified_HM_n_theta} we have
\begin{align}\label{L2_dttheta_dteta_estimate_simplified_HM_n}
&\sup_{t\in[0,T]}\norm{(\partial_t\theta^n(t),\partial_t\eta^n(t))}_{\nLpP{2}} \leq H_0(T)\norm{(\partial_z\theta^0,\partial_z\eta^0)}_{\nLpP{2}} \notag \\
&\qquad \qquad +  CLJ_0(T)\norm{(\theta^0,\eta^0)}_{\nLpP{\infty}}\norm{(\nabla_h\theta^0,\nabla_h\eta^0)}_{\nLzLphP{\infty}{4}}.
\end{align}
{\bf{Step 2}}: Convergence of the sequence $\{(\theta^n,\eta^n)\}_{n=0}^\infty$.\\
We will show that the sequence $\{(\theta^n,\eta^n)\}_{n=0}^\infty$ is a Cauchy sequence in $C([0,T];\nLpP{2})$.
\begin{claim}\label{claim_3_simplified_3D_HM}
For every $n = 0, 1, 2 ,\ldots,$
\begin{align*}
&\norm{(\theta^{n+1}-\theta^n,\eta^{n+1}-\eta^n)}_{\nLpP{2}}\leq \\
&\qquad \qquad \qquad \frac{1}{n!}\norm{(\theta^{0},\eta^{0})}_{\nLpP{2}}\left(\frac{2C_3}{L^{1/2}}K_0(T)t\norm{(\nabla_h\theta^0,\nabla_h\eta^0)}_{\nLzLphP{\infty}{4}}\right)^n,
\end{align*}
for some positive dimensionless constant $C_3$.
\end{claim}
{\it{Proof of Claim 4.3:}}
We notice that $(\theta^{n+1}-\theta^n)$ satisfies
\begin{subequations}
\begin{align*}
\pd{(\theta^{n+1}-\theta^n)}{t} + (\bu^{n}\cdot\nabla_h)\theta^{n+1} -(\bu^{n-1}\cdot\nabla_h)\theta^{n} - U_0\pd{(\theta^{n+1}-\theta^n)}{z} &=0,\\
\bu^{n-1} = K\convh \omega^ {n-1}, \, \bu^{n} = K\convh \omega^ {n}, \; \omega^{n-1} = \frac{1}{2L} (\theta^{n-1}-\eta^{n-1}),\, \omega^{n} &= \frac{1}{2L} (\theta^{n}-\eta^{n}),\\
\theta^{n+1}(0;\bx) = \theta^n(0;\bx) &= \theta^0(\bx).
\end{align*}
\end{subequations}
This can be rewritten as:
\begin{align*}
\pd{(\theta^{n+1}-\theta^n)}{t} + ((\bu^{n}-\bu^{n-1})\cdot\nabla_h)\theta^{n+1} + (\bu^{n-1}\cdot\nabla_h)(\theta^{n+1}-\theta^n) - U_0\pd{(\theta^{n+1}-\theta^n)}{z} &=0,
\end{align*}
and, taking the $\LpP{2}$ inner product with $(\theta^{n+1}-\theta^n)$ yields:
\begin{align*}
\frac{1}{2}\od{\norm{\theta^{n+1}-\theta^n}_{\LpP{2}}^2}{t} + \left< (\bu^{n}-\bu^{n-1})\cdot\nabla_h\theta^{n+1}, \theta^{n+1}-\theta^n  \right> = 0.
\end{align*}
Thus, by Corollary \ref{nonlinearity_estimate_simplified_3D_HM} we have,
\begin{align*}
\frac{1}{2}\od{\norm{\theta^{n+1}-\theta^n}_{\LpP{2}}^2}{t} &\leq \left| \left< ((\bu^{n}-\bu^{n-1})\cdot\nabla_h)\theta^{n+1}, \theta^{n+1}-\theta^n  \right>\right |\\
& \leq C_3 L^{1/2}\norm {\nabla_h \theta^{n+1}}_{\LzLphP{\infty}{4}} \norm{\omega^{n}-\omega^{n-1}}_{\LpP{2}} \norm{\theta^{n+1}-\theta^n}_{\LpP{2}},
\end{align*}
for some positive dimensionless constant $C_3$. Similar argument with $\eta^{n+1}-\eta^n$ will show,
\begin{align*}
\frac{1}{2}\od{\norm{\eta^{n+1}-\eta^n}_{\LpP{2}}^2}{t} & \leq C_3L^{1/2} \norm {\nabla_h \theta^{n+1}}_{\LzLphP{\infty}{4}} \norm{\omega^{n}-\omega^{n-1}}_{\LpP{2}} \norm{\eta^{n+1}-\eta^n}_{\LpP{2}}.
\end{align*}
Recall that
$\norm{\omega^{n}-\omega^{n-1}}_{\LpP{2}}\leq \frac{1}{L}\norm{(\theta^{n}-\theta^{n-1},\eta^{n}-\eta^{n-1})}_{\nLpP{2}}$. Thanks to \eqref{W14_estimate_linear_simplified_HM_n_theta_eta}, we have:
\begin{align}
\label{pd_L2_Cauchy_simplified_HM_n_theta_eta}
&\od{\norm{(\theta^{n+1}-\theta^n,\eta^{n+1}-\eta^n)}_{\nLpP{2}}}{t} \notag \\
&\quad  \leq \frac{C_3}{L^{1/2}}K_0(T)\norm{(\nabla_h\theta^0,\nabla_h\eta^0)}_{\nLzLphP{\infty}{4}}\norm{\omega^{n}(t)-\omega^{n-1}(t)}_{\LpP{2}}\notag \\
& \quad  \leq \frac{C_3}{L^{1/2}}K_0(T)\norm{(\nabla_h\theta^0,\nabla_h\eta^0)}_{\nLzLphP{\infty}{4}}\norm{(\theta^{n}-\theta^{n-1},\eta^{n}-\eta^{n-1})}_{\nLpP{2}}.
\end{align}
Notice that $\theta^{n+1}(0;\bx)-\theta^n(0;\bx) = \eta^{n+1}(0;\bx)-\eta^n(0;\bx)=0$, and notice that $\norm{(\theta^1-\theta^0,\eta^1-\eta^0)}_{\nLpP{2}} \leq 2\norm{(\theta^0,\eta^0)}_{\nLpP{2}} \,\, \forall n = 0, 1, 2 ,\ldots$. Integrating inequality \eqref{pd_L2_Cauchy_simplified_HM_n_theta_eta}, for $n =1$, with respect to time implies:
\begin{align*}
\norm{(\theta^{2}(t)-\theta^1(t),\eta^{2}(t)-\eta^1(t))}_{\nLpP{2}}
&\leq \frac{2C_3}{L^{1/2}}K_0(T)t\norm{(\nabla_h\theta^0,\nabla_h\eta^0)}_{\nLzLphP{\infty}{4}}  \norm{(\theta^{0},\eta^{0})}_{\nLpP{2}},
\end{align*}
for all $t\in[0,T]$. Estimate \eqref{pd_L2_Cauchy_simplified_HM_n_theta_eta}, for $n=2$, together with the above inequality give
\begin{align*}
&\od{\norm{(\theta^{3}-\theta^2,\eta^{3}-\eta^2)}_{\nLpP{2}}}{t} \notag \\
& \qquad \qquad \leq \frac{C_3}{L^{1/2}}K_0(T)\norm{(\nabla_h\theta^0,\nabla_h\eta^0)}_{\nLzLphP{\infty}{4}}\norm{(\theta^{2}-\theta^{1},\eta^{2}-\eta^{1})}_{\nLpP{2}}\notag \\
& \qquad \qquad \leq \norm{(\theta^{0},\eta^{0})}_{\nLpP{2}}\left(\frac{2C_3}{L^{1/2}}K_0(T)\norm{(\nabla_h\theta^0,\nabla_h\eta^0)}_{\nLzLphP{\infty}{4}}\right)^2 t .
\end{align*}
Integrating the inequality yields:
\begin{align*}
\norm{(\theta^{3}-\theta^2,\eta^{3}-\eta^2)}_{\nLpP{2}}\leq \frac{1}{2} \norm{(\theta^{0},\eta^{0})}_{\nLpP{2}}\left(\frac{2C_3}{L^{1/2}}K_0(T)t\norm{(\nabla_h\theta^0,\nabla_h\eta^0)}_{\nLzLphP{\infty}{4}}\right)^2 .
\end{align*}
Proceeding by induction, suppose that :
\begin{align}\label{induction_step_claim_3_simplified_3d_HM}
&\norm{(\theta^{n+1}-\theta^n,\eta^{n+1}-\eta^n)}_{\nLpP{2}} \notag\\
&\qquad \qquad \qquad \leq \frac{1}{n!}\norm{(\theta^{0},\eta^{0})}_{\nLpP{2}}\left(\frac{2C_3}{L^{1/2}}K_0(T)t\norm{(\nabla_h\theta^0,\nabla_h\eta^0)}_{\nLzLphP{\infty}{4}}\right)^n.
\end{align}
Then by \eqref{pd_L2_Cauchy_simplified_HM_n_theta_eta}, and the induction assumption \eqref{induction_step_claim_3_simplified_3d_HM} we obtain
\begin{align*}
&\od{\norm{(\theta^{n+2}-\theta^{n+1},\eta^{n+2}-\eta^{n+1})}_{\nLpP{2}}}{t} \notag \\
& \qquad \qquad \leq C_3K_0(T)\norm{(\nabla_h\theta^0,\nabla_h\eta^0)}_{\nLzLphP{\infty}{4}}\norm{(\theta^{n+1}-\theta^{n},\eta^{n+1}-\eta^{n})}_{\nLpP{2}}\notag \\
& \qquad \qquad \leq \frac{1}{n!}\norm{(\theta^{0},\eta^{0})}_{\nLpP{2}}\left(\frac{2C_3}{L^{1/2}}K_0(T)\norm{(\nabla_h\theta^0,\nabla_h\eta^0)}_{\nLzLphP{\infty}{4}}\right)^{n+1}t^n.
\end{align*}
Integrating the inequality implies:
\begin{align*}
&\norm{(\theta^{n+2}-\theta^{n+1},\eta^{n+2}-\eta^{n+1})}_{\nLpP{2}} \\
&\qquad \qquad \leq \frac{1}{(n+1)!}\norm{(\theta^{0},\eta^{0})}_{\nLpP{2}}\left(\frac{2C_3}{L^{1/2}}K_0(T)t\norm{(\nabla_h\theta^0,\nabla_h\eta^0)}_{\nLzLphP{\infty}{4}}\right)^{n+1}.
\end{align*}
This, in turn, proves Claim \ref{claim_3_simplified_3D_HM}.
Now, let $m>n$, we conclude from Claim \ref{claim_3_simplified_3D_HM} that:
\begin{align*}
&\sup_{t\in[0,T]}\norm{(\theta^{m}(t)-\theta^n(t),\eta^{m}(t)-\eta^n(t))}_{\nLpP{2}}\\
&\qquad \qquad \leq \norm{(\theta^{0},\eta^{0})}_{\nLpP{2}}\sum_{k=n+1}^m \frac{1}{k!}\left(\frac{2C_3}{L^{1/2}}K_0(T)T\norm{(\nabla_h\theta^0,\nabla_h\eta^0)}_{\nLzLphP{\infty}{4}}\right)^{k}.
\end{align*}
This implies that $\{(\theta^n,\eta^n)\}_{n=1}^\infty$ is a Cauchy sequence in $C([0,T];\nLpP{2})$ since
$$\sum_{k=0}^\infty\frac{1}{k!}\left(\frac{2C_3}{L^{1/2}}K_0(T)T\norm{(\nabla_h\theta^0,\nabla_h\eta^0)}_{\nLzLphP{\infty}{4}}\right)^{k}= e^{\frac{2C_3}{L^{1/2}}K_0(T)T\norm{(\nabla_h\theta^0,\nabla_h\eta^0)}_{\nLzLphP{\infty}{4}}},$$
Then there exists  $({ \theta}, { \eta}) \in C([0,T];\nLpP{2})$ such that:
\begin{align}
\label{strong_convergence_theta_eta}
(\theta^n,\eta^n) \strong ({ \theta},{ \eta}) \quad \quad \quad \text{in}\, \, \, C([0,T];\nLpP{2}),
\end{align}
as $n\strong \infty$; and hence,
\begin{align}\label{L2_convergence_norm_theta}
\norm{({\theta}(t),{\eta}(t)}_{\nLpP{2}} = \lim_{n\rightarrow0}\norm{(\theta^n(t),\eta^n(t))}_{\nLpP{2}}= \norm{(\theta^0,\eta^0)}_{\nLpP{2}},
\end{align}
for all $t\in[0,T]$. Define ${ \bu} := K\convh{ \omega} = \frac{1}{2L} K\convh({ \theta}-{ \eta})$, then by the elliptic regularity estimate \eqref{elliptic_regularity_Yudovich},
\begin{align*}
&\norm{\bu^n(t;z)-{\bu}(t;z)}_{\LphP{2}}^2 \leq CL \norm{\omega^n(t;z)-{\omega}(t;z)}_{\LphP{2}}^2 \notag \\
&\qquad \qquad \qquad  \leq C \norm{(\theta^n(t;z)-{\theta}(t;z),\eta^n(t;z)-{\eta}(t;z))}_{\nLphP{2}}^2,
\end{align*}
for every $n=1,2,..$. After we integrate with respect to $z$ over $[0,L]$ we obtain:
\begin{align}\label{4_40_simplified-3d_HM}
\norm{\bu^n(t)-{\bu}(t)}_{\LpP{2}} \leq  C \norm{(\theta^n(t)-{ \theta}(t),\eta^n(t)-{\eta}(t))}_{\nLpP{2}},
\end{align}
for all $n=1,2,\ldots$. Thus, we conclude that
\begin{align}
\label{strong_convergence_u}
{\bu^n} \strong {{\bu}} \quad \quad \quad \text{in}\, \, \, C([0,T];\LpP{2}),
\end{align}
as $n\strong \infty$, and by \eqref{4_40_simplified-3d_HM}
\begin{align}\label{L2_norm_convergence_u}
\sup_{t\in[0,T]}\norm{{\bu}(t)}_{\LpP{2}}&\leq C\sup_{t\in[0,T]}\norm{({\theta}(t),{\eta}(t))}_{\nLpP{2}} = C\norm{(\theta^0,\eta^0)}_{\nLpP{2}}.
\end{align}
{\bf{Step 3}}: Passing in the limit to a weak solution $({ \theta}, { \eta})$.\\
For each $n = 0, 1, 2 ,\ldots$, $(\theta^n,\eta^n)$ is a classical solution of \eqref{simplified_3D_HM_n}, thus $(\theta^n,\eta^n)$ is also a weak solution of \eqref{simplified_3D_HM_n} in the sense of Definition \ref{weak_solution_definition_simplified_3D_HM}, then for any $\chi_{\bm}, \gamma_{\bk} \in C^\infty([0,T])$, with $\chi_{\bm}(T)=0$ and $\gamma_{\bk}(T)=0$,
\begin{align*}
&-\int_0^T\left<\theta^n(s),e^{2\pi i\bm\cdot \bx}\right>\chi^{'}_{\bm}(s)\,ds - \int_0^T\left<\bu\theta^n,\nabla_h e^{2\pi i\bm\cdot \bx}\right>\chi_{\bm}(s)\,ds\notag \\
&\qquad \qquad \qquad \qquad \qquad +U_0\int_0^T\left<\theta^n(s),\pd{e^{2\pi i\bm\cdot \bx}}{z}\right>\chi_{\bm}(s)\,ds =\left<\theta^0,e^{2\pi i\bm\cdot \bx}\right>\chi_{\bm}(0),
\end{align*}
\begin{align*}
&-\int_0^T\left<\eta^n(s),e^{2\pi i\bk\cdot \bx}\right>\gamma^{'}_{\bk}(s)\,ds - \int_0^T\left<\bu\eta^n,\nabla_h e^{2\pi i\bk\cdot \bx}\right>\gamma_{\bk}(s)\,ds \notag \\
&\qquad \qquad \qquad \qquad \qquad -U_0\int_0^T\left<\eta^n(s),\pd{e^{2\pi i\bk\cdot \bx}}{z}\right>\gamma_{\bk}(s)\,ds =\left<\eta^0,e^{2\pi i\bk\cdot \bx}\right>\gamma_{\bk}(0).
\end{align*}
for all $\bm,\bk \in\nZ^2$. The strong convergence \eqref{strong_convergence_theta_eta}, obtained in Step 2 of the proof, implies that

\begin{align}
\label{linear_convergence_1_theta}
\int_0^T\left<\theta^n(s),e^{2\pi i\bm\cdot \bx}\right>\chi^{'}_{\bm}(s)\,ds & \strong \int_0^T\left<{{\theta}}(s),e^{2\pi i\bm\cdot \bx}\right>\chi^{'}_{\bm}(s)\,ds, \\
\label{linear_convergence_1_eta}
\int_0^T\left<\eta^n(s),e^{2\pi i\bk\cdot \bx}\right>\gamma^{'}_{\bk}(s)\,ds & \strong \int_0^T\left<{{\eta}}(s),e^{2\pi i\bk\cdot \bx}\right>\gamma^{'}_{\bk}(s)\,ds, \\
\label{linear_convergence_2_theta}
\int_0^T\left<\theta^n(s),\pd{e^{2\pi i\bm\cdot \bx}}{z}\right>\chi_{\bm}(s)\,ds & \strong  \int_0^T\left<{{\theta}}(s),\pd{e^{2\pi i\bm\cdot \bx}}{z}\right>\chi_{\bm}(s)\,ds,\\
\label{linear_convergence_2_eta}
\int_0^T\left<\eta^n(s),\pd{e^{2\pi i\bk\cdot \bx}}{z}\right>\gamma_{\bk}(s)\,ds & \strong  \int_0^T\left<{{\eta}}(s),\pd{e^{2\pi i\bk\cdot \bx}}{z}\right>\gamma_{\bk}(s)\,ds,
\end{align}
as $n \strong \infty$. It remains to show the convergence for the nonlinear terms. By adding and subtracting $\left<\bu^{n-1}{{\theta}},\nabla_he^{2\pi i\bm\cdot \bx}\right>$ and $\left<\bu^{n-1}{{\eta}},\nabla_he^{2\pi i\bk\cdot \bx}\right>$, respectively, we get
\begin{align*}
&\int_0^T\left<\bu^{n-1}\theta^n,\nabla_h e^{2\pi i\bm\cdot \bx}\right>\chi_{\bm}(s)\,ds - \int_0^T\left<{{\bu}}{{\theta}},\nabla_h e^{2\pi i\bm\cdot \bx}\right>\chi_{\bm}(s)\,ds =\notag \\
 & \int_0^T\left<\bu^{n-1}(\theta^n - {{\theta}}) ,\nabla_h e^{2\pi i\bm\cdot \bx}\right>\chi_{\bm}(s)\,ds
+ \int_0^T\left<(\bu^{n-1}-{{\bu}}) \theta^n,\nabla_h e^{2\pi i\bm\cdot \bx}\right>\chi_{\bm}(s)\,ds,
\end{align*}
and,
\begin{align*}
&\int_0^T\left<\bu^{n-1}\eta^n,\nabla_h e^{2\pi i\bk\cdot \bx}\right>\gamma_{\bk}(s)\,ds - \int_0^T\left<{{\bu}}{{\eta}},\nabla_h e^{2\pi i\bk\cdot \bx}\right>\gamma_{\bk}(s)\,ds =\notag \\
 & \int_0^T\left<\bu^{n-1}(\eta^n - {{\eta}}) ,\nabla_h e^{2\pi i\bk\cdot \bx}\right>\gamma_{\bk}(s)\,ds
+ \int_0^T\left<(\bu^{n-1}-{{\bu}}) \theta^n,\nabla_h e^{2\pi i\bk\cdot \bx}\right>\gamma_{\bk}(s)\,ds.
\end{align*}
By H\"older's inequality, we have:
\begin{align*}
&\abs{\int_0^T\left<\bu^{n-1}(\theta^n - {{\theta}}) ,\nabla_h e^{2\pi i\bm\cdot \bx}\right>\chi_\bm(s)\,ds} \leq \notag \\
&\quad \int_0^T \norm{\nabla_h e^{2\pi i\bm\cdot \bx}}_{\LpP{\infty}}\norm{\bu^{n-1}(s)}_{\LpP{2}}\norm{\theta^n(t) - {{\theta}(t)}}_{\LpP{2}}|\chi_\bm(s)|\,ds.
\end{align*}
Since $\abs{\chi_\bm(t)}\leq C$, for every $t\in[0,T]$, and $\norm{\nabla_h e^{2\pi i\bm\cdot \bx}}_{\LpP{\infty}}\leq \frac{C}{L}$, then from the above and \eqref{L2_norm_convergence_u} we obtain
\begin{align*}
&\abs{\int_0^T\left<\bu^{n-1}(\theta^n - {{\theta}}) ,\nabla_h e^{2\pi i\bm\cdot \bx}\right>\chi_\bm(s)\,ds} \leq \frac{C}{L}T\norm{(\theta^0,\eta^0)}_{\nLpP{2}}\sup_{t\in[0,T]}\norm{\theta^n(t) - {{\theta}(t)}}_{\LpP{2}}.
\end{align*}
Similar argument, and the use of \eqref{L2_convergence_norm_theta}, will imply
\begin{align*}
&\abs{\int_0^T\left<(\bu^{n-1}-{{\bu}}) \theta,\nabla_h e^{2\pi i\bm\cdot \bx}\right>\chi_\bm(s)\,ds} \leq\frac{C}{L}T\norm{(\theta^0,\eta^0)}_{\nLpP{2}}\sup_{t\in[0,T]}\norm{\bu^{n-1}(t) - {\bu}(t)}_{\LpP{2}}.
\end{align*}
We can conclude, similarly, that
\begin{align*}
&\abs{\int_0^T\left<\bu^{n-1}(\eta^n - {{\eta}}) ,\nabla_h e^{2\pi i\bk\cdot \bx}\right>\gamma_\bk(s)\,ds} \leq  \frac{C}{L}T\norm{(\theta^0,\eta^0)}_{\nLpP{2}}\sup_{t\in[0,T]}\norm{\eta^n(t) - {{\eta}(t)}}_{\LpP{2}},
\end{align*}
and
\begin{align*}
&\abs{\int_0^T\left<(\bu^{n-1}-{{\bu}}) \eta,\nabla_h e^{2\pi i\bk\cdot \bx}\right>\gamma_\bk(s)\,ds} \leq \frac{C}{L}T\norm{(\theta^0,\eta^0)}_{\nLpP{2}}\sup_{t\in[0,T]}\norm{\bu^{n-1}(t) - {{\bu}(t)}}_{\LpP{2}}.
\end{align*}
Thus the convergence,
\begin{align}
\label{nonlinearity_convergence_theta_1}
\int_0^T\left<\bu^{n-1}(\theta^n - {{\theta}}) ,\nabla_h e^{2\pi i\bm\cdot \bx}\right>\chi_\bm(s)\,ds &\strong 0, \\
\label{nonlinearity_convergence_theta_2}
\int_0^T\left<(\bu^{n-1}-{{\bu}}) \theta^n,\nabla_h e^{2\pi i\bm\cdot \bx}\right>\chi_\bm(s)\,ds &\strong 0,\\
\label{nonlinearity_convergence_eta_1}
\int_0^T\left<\bu^{n-1}(\eta^n - {{\eta}}) ,\nabla_h e^{2\pi i\bk\cdot \bx}\right>\gamma_\bk(s)\,ds &\strong 0, \\
\label{nonlinearity_convergence_eta_2}
\int_0^T\left<(\bu^{n-1}-{{\bu}}) \eta^n,\nabla_h e^{2\pi i\bk\cdot \bx}\right>\gamma_\bk(s)\,ds &\strong 0,
\end{align}
as $n \strong \infty$, follows immediately by \eqref{strong_convergence_theta_eta} and \eqref{strong_convergence_u}.
The convergence in \eqref{linear_convergence_1_theta}, \eqref{linear_convergence_2_theta}, \eqref{linear_convergence_1_eta}, \eqref{linear_convergence_2_eta}, \eqref{nonlinearity_convergence_theta_1}, \eqref{nonlinearity_convergence_theta_2},  \eqref{nonlinearity_convergence_eta_1} and \eqref{nonlinearity_convergence_eta_2} show that $({{\theta}},{{\eta}})$ satisfies:
\begin{align}\label{solution_theta_distribution_simplified_3D_HM}
&\int_0^T\left<{{\theta}}(s),e^{2\pi i\bm\cdot \bx}\right>\chi_\bm^{'}(s)\,ds - \int_0^T\left<{{\bu}}{{\theta}},\nabla_h e^{2\pi i\bm\cdot \bx}\right>\chi_\bm(s)\,ds \notag \\
&\qquad \qquad \qquad +U_0\int_0^T\left<{{\theta}}(s), \pd{e^{2\pi i\bm\cdot \bx}}{z}\right>\chi_\bm(s)\,ds=\left<\theta^0,e^{2\pi i\bm\cdot \bx}\right>\chi_\bm(0),
\end{align}
\begin{align}\label{solution_eta_distribution_simplified_3D_HM}
&\int_0^T\left<{{\eta}}(s),e^{2\pi i\bk\cdot \bx}\right>\gamma_\bk^{'}(s)\,ds - \int_0^T\left<{{\bu}}{{\eta}},\nabla_h e^{2\pi i\bk\cdot \bx}\right>\gamma_\bk(s)\,ds \notag \\
&\qquad \qquad \qquad -U_0\int_0^T\left<{{\eta}}(s), \pd{e^{2\pi i\bk\cdot \bx}}{z}\right>\gamma_\bk(s)\,ds=\left<\eta^0,e^{2\pi i\bk\cdot \bx}\right>\gamma_\bk(0),
\end{align}
where ${ \bu} = K\convh{ \omega} = \frac{1}{2L}K\convh({ \theta} - { \eta})$.

Now, we will check that ${ \theta}(0) = \theta^0$ and ${ \eta}(0) = \eta^0$. As we showed before the statement of Theorem \ref{global_existence_simplified_HM_classical}, for any $\psi_1,\psi_2 \in L^1([0,T];\LpP{2})$,
\begin{align*}
\int_0^T\left<\pd{\theta}{t}(t),\psi_1(t)\right>\,dt &=\int^0_T\left<-\nabla_h\cdot({ \bu} { \theta})(t), \psi_1(t)\right>\,dt + U_0\int_0^T\left<\pd{\theta}{z}(t), \psi_1(t)\right>\,dt,\\
\int_0^T\left<\pd{\eta}{t}(t),\psi_2(t)\right>\,dt &= \int_0^T\left<-\nabla_h\cdot({ \bu} { \eta})(t), \psi_2(t)\right>\,dt - U_0\int_0^T\left<\pd{\eta}{z}(t), \psi_2(t)\right>\,dt.
\end{align*}
Choosing $\psi_1(t; \bx) = e^{2\pi i\bm\cdot \bx}\chi_{\bm}(t),\psi_2(t; \bx) = e^{2\pi i\bk\cdot \bx}\gamma_{\bk}(t)$ and using integration by parts and  integrating with respect to time imply that
\begin{align*}
\int_0^T\left<{{\theta}}(s),e^{2\pi i\bm\cdot \bx}\right>\chi_\bm^{'}(s)\,ds &- \int_0^T\left<{{\bu}}{{\theta}},\nabla_h e^{2\pi i\bm\cdot \bx}\right>\chi_\bm(s)\,ds \notag \\
&+U_0\int_0^T\left<{{\theta}}(s), \pd{e^{2\pi i\bm\cdot \bx}}{z}\right>\chi_\bm(s)\,ds=\left<{ \theta}(0),e^{2\pi i\bm\cdot \bx}\right>\chi_\bm(0),
\end{align*}
and
\begin{align*}
\int_0^T\left<{{\eta}}(s),e^{2\pi i\bk\cdot \bx}\right>\gamma_\bk^{'}(s)\,ds &- \int_0^T\left<{{\bu}}{{\eta}},\nabla_h e^{2\pi i\bk\cdot \bx}\right>\gamma_\bk(s)\,ds \notag \\
&-U_0\int_0^T\left<{{\eta}}(s), \pd{e^{2\pi i\bk\cdot \bx}}{z}\right>\gamma_\bk(s)\,ds=\left<{ \eta}(0),e^{2\pi i\bk\cdot \bx}\right>\gamma_\bk(0).
\end{align*}
By comparison with \eqref{solution_theta_distribution_simplified_3D_HM} and \eqref{solution_eta_distribution_simplified_3D_HM}, we see that

\begin{align*}
\left<{ \theta}(0),e^{2\pi i\bm\cdot \bx}\right>\chi_\bm(0)& = \left<\theta^0,e^{2\pi i\bm\cdot \bx}\right>\chi_\bm(0), \\
\left<{ \eta}(0),e^{2\pi i\bk\cdot \bx}\right>\gamma_\bk(0) &= \left<\eta^0,e^{2\pi i\bk\cdot \bx}\right>\gamma_\bk(0).
\end{align*}
for all $\bm, \bk\in\nZ^2$ and any function $\chi_\bm$ and $\gamma_\bk$. We can choose $\chi_\bm, \gamma_\bk$ such that $\chi_\bm(0), \gamma_\bk(0) \not = 0$ and therefore
$$
\left<{ \theta}(0) - \theta^0, e^{2\pi i\bm\cdot \bx}\right> = \left<{ \eta}(0) - \eta^0, e^{2\pi i\bk\cdot \bx}\right> =0$$
This implies
\begin{align}\label{theta_0_eta_0_equality_simplified_3D_HM}
{ \theta}(0) = \theta^0, \qquad { \eta}(0) = \eta^0.
\end{align}

By \eqref{Linfty_estimate_linear_simplified_HM_n_theta_eta}, the sequence $\{(\theta^n, \eta^n)\}_{n=0}^\infty$ is uniformly bounded in $L^\infty([0,T];\nLpP{p})$. Estimate \eqref{W14_estimate_linear_simplified_HM_n_theta_eta} shows that
$\{(\nabla_h\theta^n,\nabla_h \eta^n)\}_{n=0}^\infty$ is uniformly bounded in $L^\infty([0,T];\\ \nLzLphP{\infty}{4})$. Estimate \eqref{L2_dztheta_dzeta_estimate_simplified_HM_n} shows that $\{(\partial_z\theta^n,\partial_z\eta^n)\}_{n=0}^\infty$ is uniformly bounded in $L^\infty([0,T];\nLpP{2})$, and the estimate \eqref{L2_dttheta_dteta_estimate_simplified_HM_n} shows that $\{(\partial_t\theta^n,\partial_t\eta^n\}$ is uniformly bounded in $L^\infty([0,T];\nLpP{2})$. We can pass, by the Banach-Alaoglu theorem, to a weak-$\ast$ convergent subsequence (that we will still denote by $\{(\theta^n,\eta^n)\}_{n=0}^\infty$) in $L^\infty([0,T];\nLpP{\infty})$ such that $\{(\nabla_h\theta^n,\nabla_h \eta^n)\}_{n=0}^\infty$ is weak-$\ast$ in $L^\infty([0,T];\nLzLphP{\infty}{4})$, $\{(\partial_z\theta^n,\partial_z\eta^n)\}_{n=0}^\infty$ and $\{(\partial_t\theta^n,\partial_t\eta^n)\}_{n=0}^\infty$ are also weak-$\ast$ convergent in $L^\infty([0,T];\\ \nLpP{2})$.
Then, by the uniqueness of the limit, we have:
\begin{align*}
(\theta^n,\eta^n) &\weakstar ({{\theta}}, { \eta})\qquad\qquad\quad \text {in} \,\,\,L^\infty([0,T];\nLpP{p}), \\
(\nabla_h\theta^n,\nabla_h\eta^n) & \weakstar ( \nabla_h{{\theta}},\nabla_h{ \eta}) \qquad \; \text {in} \,\,\,L^\infty([0,T];\nLzLphP{\infty}{4}), \\
(\partial_z \theta^n,\partial_z\eta^n) &\weakstar  (\partial_z{{\theta}}, \partial_z{ \eta}) \qquad \quad \text{in} \,\,\,L^\infty([0,T];\nLpP{2}),\\
(\partial_t \theta^n,\partial_t\eta^n) &\weakstar  (\partial_t{{\theta}},\partial_t{ \eta}) \qquad \quad\, \text{in} \,\,\,L^\infty([0,T];\nLpP{2}),
\end{align*}
as $n \strong \infty$ for any $p\in[1,\infty]$. Then, it follows that:
\begin{align*}
\sup_{t\in[0,T]}\norm{({\theta}(t),{ \eta}(t))}_{\nLpP{p}} &\leq \norm{(\theta^0,\eta^0)}_{\nLpP{p}},
\end{align*}
for any $p\in[1,\infty]$, and by \eqref{theta_0_eta_0_equality_simplified_3D_HM} we conclude that
\begin{align*}
\sup_{t\in[0,T]}\norm{({\theta}(t),{ \eta}(t))}_{\nLpP{p}} &= \norm{(\theta^0,\eta^0)}_{\nLpP{p}},
\end{align*}
for any $p\in[1,\infty]$. Also,
\begin{align*}
\sup_{t\in[0,T]}\norm{(\nabla_h{\theta}(t),\nabla_h{\eta}(t))}_{\nLzLphP{\infty}{4}} \leq {J_0(T)}\norm{(\nabla_h\theta^0,\nabla_h\eta^0)}_{\nLzLphP{\infty}{4}},
\end{align*}
\begin{align*}
\sup_{t\in[0,T]}\norm{(\partial_z{ \theta}(t), \partial_z{ \eta}(t))}_{\nLpP{2}} \leq {H_0(T)}\norm{(\partial_z\theta^0,\partial_z\eta^0)}_{\nLpP{2}},
\end{align*}
\begin{align*}
&\sup_{t\in[0,T]}\norm{(\partial_t{ \theta}(t),\partial_t{ \eta}(t))}_{\nLpP{2}} \leq H_0(T)\norm{(\partial_z\theta^0,\partial_z\eta^0)}_{\nLpP{2}} \notag \\
&\qquad \qquad +  CLJ_0(T)\norm{(\theta^0,\eta^0)}_{\nLpP{\infty}}\norm{(\nabla_h\theta^0,\nabla_h\eta^0)}_{\nLzLphP{\infty}{4}}.
\end{align*}

We just showed that $({ \theta},{ \eta})$ is a weak solution of \eqref{simplified_3D_HM} in the sense of Definition \ref{weak_solution_definition_simplified_3D_HM}. Finally, by a similar argument as in the proofs of Theorem \ref{global_existence_uniqueness_linear_simplified_HM} and Proposition \ref{nabla_h_u_Linfty_estimate_simplified_3D_HM}, we have
\begin{align*}
\sup_{t\in[0,T]}\norm{\nabla_h {\bu}(t)}_{\LpP{\infty}}&\leq \frac{C}{L}\norm{(\theta^0,\eta^0)}_{\nLpP{\infty}}\log \left(e^2+CL^2J_0(T)\frac{\norm{(\nabla_h\theta^0,\nabla_h\eta^0)}_{\nLzLphP{\infty}{4}}}{\norm{(\theta^0,\eta^0)}_{\nLpP{2}}}\right).
\end{align*}
The uniqueness of the weak solution $({\theta}, { \eta})$ is an immediate consequence of the next Corollary \ref{cont_dep_initial_data_simplified_HM}. This completes the proof.
\end{proof}

\begin{corollary}
\label{cont_dep_initial_data_simplified_HM}
Let $(\theta_1^0,\eta_1^0), (\theta_2^0,\eta_2^0)\in \nLpP{2}$, $(\nabla_h\theta_1^0,\nabla_h\eta_1^0), (\nabla_h\theta_2^0,\nabla_h\eta_2^0)\in \nLzLphP{\infty}{4} $ and assume the existence of corresponding two weak solutions  $(\theta_1,\eta_1), (\theta_2,\eta_2)$ of \eqref{simplified_3D_HM}, such that the weak solution $(\theta_2,\eta_2)$ satisfies
\begin{align}\label{J_0_2_assumption_simplified_3D_HM}
\sup_{t\in[0,T]}\norm{(\nabla_h\theta_2(t),\nabla_h\eta_2(t))}_{\nLzLphP{\infty}{4}} \leq J_{0,2}\norm{(\nabla_h\theta_2^0,\nabla_h\eta_2^0)}_{\nLzLphP{\infty}{4}}.
\end{align}
for some positive constant $J_{0,2}$. Then,
\begin{align}
\label{cont_dep_inequality_simplified_HM}
&\sup_{t\in[0,T]}\norm{(\theta_1(t)-\theta_2(t),\eta_1(t)-\eta_2(t))}_{\nLpP{2}}\notag \\
& \qquad\qquad \leq Ce^{\frac{C}{L^{1/2}}J_{0,2}T \norm{(\nabla_h\theta_2^0,\nabla_h\eta_2^0)}_{\nLzLphP{\infty}{4}}}\norm{(\theta_1^0-\theta_2^0,\eta_1^0-\eta_2^0)}_{\nLpP{2}}.
\end{align}
\end{corollary}
\begin{proof}
Let us assume that $(\theta_1,\eta_1)$ and $(\theta_2,\eta_2)$ are two weak solutions with initial data, $(\theta_1^0,\eta_1^0)$ and $(\theta_2^0,\eta_2^0)$, respectively, then the differences, $\theta= \theta_1-\theta_2$ and $\eta = \eta_1-\eta_2$ satisfy
\begin{subequations}\label{difference_theta_eta_simplified_3D_HM}
\begin{align}
\pd{\theta}{t} + \nabla_h\cdot (\bu_1\theta_1) - \nabla_h\cdot(\bu_2\theta_2) -U_0\pd{\theta}{z} & = 0,\\
\pd{\eta}{t} + \nabla_h\cdot (\bu_1\eta_1) - \nabla_h\cdot(\bu_2\eta_2) +U_0\pd{\eta}{z} & = 0,
\end{align}
\end{subequations}
in $L^\infty([0,T];\LpP{2})$, where $\bu_1 = \frac{1}{2L}K\convh(\theta_1-\eta_1)$ and $\bu_2 = \frac{1}{2L}K\convh(\theta_2-\eta_2)$.
Recall that $\theta, \eta \in C([0,T];\LpP{2})$ and $\pd{\theta}{t}, \pd{\eta}{t} \in L^\infty([0,T];\LpP{2})$. Thanks to Lemma 1.2 (\cite{Temam_2001_Th_Num}, page 176), the equations below hold
\begin{align}\label{Lions_theta_eta_simplified_3D_HM}
\frac{1}{2}\od{\norm{\theta(t)}^2_{\LpP{2}}}{t} = \left<\pd{\theta}{t}(t),\theta\right> \quad \text{and} \quad \frac{1}{2}\od{\norm{\eta(t)}^2_{\LpP{2}}}{t} = \left<\pd{\eta}{t}(t),\theta\right>.
\end{align}
We can take a.e. in $t$ the $\LpP{2}$ inner product of \eqref{difference_theta_eta_simplified_3D_HM} with $\theta(t)$ and $\eta(t)$, respectively, and get
\begin{align}
\frac{1}{2}\od{\norm{\theta(t)}^2_{\LpP{2}}}{t} &= -\left<(\bu\cdot\nabla_h) \theta_2(t),\theta(t)\right>\notag \\
& \quad \leq C L^{1/2}\norm{\nabla_h\theta_2(t)}_{\LzLphP{\infty}{4}} \norm{\omega(t)}_{\LpP{2}} \norm{\theta(t)}_{\LpP{2}},\label{theta_diff_simplified_3d_HM} \\
\frac{1}{2}\od{\norm{\eta(t)}^2_{\LpP{2}}}{t} &= -\left<((\bu(t))\cdot\nabla_h) \eta_2(t),\eta(t))\right>\notag \\
& \quad \leq C L^{1/2}\norm{\nabla_h\eta_2(t)}_{\LzLphP{\infty}{4}} \norm{\omega(t)}_{\LpP{2}} \norm{\eta(t)}_{\LpP{2}}, \label{eta_diff_simplified_3d_HM}
\end{align}
a.e. $t\in [0,T]$ where $\bu = \bu_1-\bu_2$ and $\omega = \omega_1-\omega_2$. Recall that by the use of \eqref{elliptic_regularity_Yudovich},
\begin{align*}
\norm{\omega(t)}_{\LpP{2}}\leq \frac{1}{L}\norm{\bu(t)}_{\LpP{2}}\leq C\norm{(\theta(t),\eta(t))}_{\nLpP{2}},
\end{align*}
for any $t\in[0,T]$. Therefore, summing the inequalities \eqref{theta_diff_simplified_3d_HM} and \eqref{eta_diff_simplified_3d_HM} and using the assmuption \eqref{J_0_2_assumption_simplified_3D_HM} yield
\begin{align*}
&\od{\left(\norm{\theta(t)}^2_{\LpP{2}}+\norm{\eta(t)}^2_{\LpP{2}}\right)}{t} \\&
 \leq \frac{C}{L^{1/2}}J_{0,2} \norm{(\nabla_h\theta_2^0,\nabla_h\eta_2^0)}_{\nLzLphP{\infty}{4}} \left(\norm{\theta(t)}^2_{\LpP{2}}+\norm{\eta(t)}^2_{\LpP{2}}\right),
\end{align*}
a.e. $t\in[0,T]$. Integrating the inequality with respect to time and using Gronwall's lemma, and the fact that norms are equivalent in finite dimensional spaces, we have,
\begin{align*}
&\norm{(\theta_1(t)-\theta_2(t),\eta_1(t)-\eta_2(t))}_{\nLpP{2}}\\
&\qquad \qquad \leq Ce^{\frac{C}{L^{1/2}}J_{0,2} T\norm{(\nabla_h\theta_2^0,\nabla_h\eta_2^0)}_{\nLzLphP{\infty}{4}}}\norm{(\theta_1^0-\theta_2^0,\eta_1^0-\eta_2^0)}_{\nLpP{2}},
\end{align*}
a.e. $t\in[0,T]$. This proves the corollary.
\end{proof}

\begin{theorem}
\label{global_existence_uniqueness_simplified_HM}
Let $T>0$, let $(\theta^0,\eta^0)\in\nLpP{\infty}$ with $(\nabla_h\theta^0,\nabla_h\eta^0) \in \nLzLphP{\infty}{4}$, $(\partial_z\theta^0,\partial_z\eta^0) \in \nLpP{2}$ be given. Then system \eqref{simplified_3D_HM} has a unique weak solution $(\theta,\eta)$ in the sense of Definition \ref{weak_solution_definition_simplified_3D_HM}, that satisfies estimates \eqref{estimate_a_simplified_3D_HM}--\eqref{estimate_b_simplified_3D_HM}
Furthermore,
\begin{align}
&\sup_{t\in[0,T]}\norm{\nabla_h \bu(t)}_{\LpP{\infty}}\leq \notag \\
& \frac{C}{L}\norm{(\theta^0,\eta^0)}_{\nLpP{\infty}}\log \left(e^2+CL^2J_0(T)\frac{\norm{(\nabla_h\theta^0,\nabla_h\eta^0)}_{\nLzLphP{\infty}{4}}}{\norm{(\theta^0,\eta^0)}_{\nLpP{2}}}\right),
\end{align}
where $\bu = K\convh{{\omega}}$ with ${\omega} = \frac{1}{2}({\theta}-{\eta})$. $C$ is a constant that depends only on the domain $\nT^3$, $J_0(T)$, depends on $T$ and on the domain $\nT^3$ and the norms of the initial data $(\theta^0,\eta^0)$ and are specified in \eqref{J_0_simplified_3D_HM}. Moreover, if $(\nabla_h\theta_1^0,\nabla_h\eta_1^0), (\nabla_h\theta_2^0,\nabla_h\eta_2^0) \in \nLzLphP{\infty}{4}$, $(\partial_z\theta^0_1,\partial_z\eta_1^0), (\partial_z\theta_2^0,\partial_z\eta_2^0) \in \nLpP{2}$, then the corresponding weak solutions $(\theta_1,\eta_1), (\theta_2,\eta_2)$ of system \eqref{simplified_3D_HM} satisfy:
\begin{align}
&\sup_{t\in[0,T]}\norm{(\theta_1(t)-\theta_2(t),\eta_1(t)-\eta_2(t))}_{\nLpP{2}}\\
&\qquad \qquad \leq Ce^{\frac{C}{L^{1/2}}J_{0}(T)T\norm{(\nabla_h\theta_2^0,\nabla_h\eta_2^0)}_{\nLzLphP{\infty}{4}}}\norm{(\theta_1^0-\theta_2^0,\eta_1^0-\eta_2^0)}_{\nLpP{2}}.
\end{align}
\end{theorem}

\begin{proof}
The proof is based again on the idea of construction of a sequence of approximate weak solutions $\{(\theta^\veps, \eta^\veps)\}_{\veps\in(0,L/4]}$ of \eqref{simplified_3D_HM}  and then passing to the limit to a weak solution of \eqref{simplified_3D_HM}. Without loss of generality, we will assume that $(\theta^0,\eta^0) \not= {\bf 0}$.

{\bf{Step 1}}: Construction of a sequence of approximate solutions $\{(\theta^\veps,\eta^\veps)\}_{\veps\in(0,\frac{L}{4}]}$.\\
Let $\rho(x,y,z) = \zeta(x,y)\xi(z)$ be a $C^\infty$-compactly supported mollifier, such that $\zeta(x,y)\geq0$ and $\xi(z) \geq0$ are $C^\infty$-compactly supported mollifiers, with $\int_0^L\xi(z)\, dz= 1$ and $\int_0^L\int_0^L\zeta(x,y)\, dxdy =1$. Let $\rho^\veps(x,y,x) = \zeta^\veps(x,y)\xi^\veps(z)$, where $\zeta^\veps(x,y) = \frac{1}{\veps^2}\zeta(\frac{x}{\veps},\frac{y}{\veps})$ and $\xi^\veps(z) = \frac{1}{\veps}\xi(\frac{z}{\veps})$.
By Theorem 6 (\cite{Evans_1998}, page 630), $(\theta^0(\bx)\ast \rho^\veps, \eta^0(\bx)\ast \rho^\veps)\in C^\infty_{per}(\nT^3)\times C^\infty_{per}(\nT^3) \, \forall \veps\in(0,\frac{L}{4}]$. By Young's convolution inequality:
\begin{align*}
\norm{\theta^0\ast {\rho^\veps}}_{\LpP{2}}&\leq \norm{\theta^0}_{\LpP{2}}, \qquad \quad \norm{\eta^0\ast {\rho^\veps}}_{\LpP{2}} \leq \norm{\eta^0}_{\LpP{2}}, \\\norm{\theta^0\ast {\rho^\veps}}_{\LpP{\infty}}& \leq \norm{\theta^0}_{\LpP{\infty}}, \qquad\;\; \norm{\eta^0\ast {\rho^\veps}}_{\LpP{\infty}}  \leq \norm{\eta^0}_{\LpP{\infty}},
\end{align*}
\begin{align*}
\norm{\partial_z\theta^0\ast {\rho^\veps}}_{\LpP{2}} &\leq \norm{\partial_z\theta^0}_{\LpP{2}}, \qquad \norm{\partial_z\eta^0\ast {\rho^\veps}}_{\LpP{\infty}} \leq \norm{\partial_z\eta^0}_{\LpP{\infty}}, \\
\norm{\nabla_h\theta^0\ast{\rho^\veps}}_{\LzLphP{\infty}{4}} &= {\norm{\norm{\left(\nabla_h\theta^0\ast{\zeta^\veps}\right)(z)}_{\LphP{4}}\ast{\xi^\veps}}}_{\LzP{\infty}}\leq \norm{\nabla_h\theta^0}_{\LzLphP{\infty}{4}}, \\
\norm{\nabla_h\eta^0\ast{\rho^\veps}}_{\LzLphP{\infty}{4}} &= {\norm{\norm{\left(\nabla_h\eta^0\ast{\zeta^\veps}\right)(z)}_{\LphP{4}}\ast{\xi^\veps}}}_{\LzP{\infty}}\leq \norm{\nabla_h\eta^0}_{\LzLphP{\infty}{4}}.
\end{align*}
By Theorem \ref{global_existence_simplified_HM_classical}, and the above estimates, there exists a weak solution $(\theta^\veps,\eta^\veps)$ of the system
\begin{subequations}\label{eps_simplified_3D_HM}
\begin{align}
\pd{\theta^\veps}{t} + (\bu^\veps\cdot\nabla_h)\theta^\veps - U_0\pd{\theta^\veps}{z} =0, \qquad \pd{\eta^\veps}{t} + (\bu^\veps\cdot\nabla_h)\eta^\veps + U_0\pd{\eta^\veps}{z} =0, \\
\bu^\veps = K\convh \omega^\veps, \, \omega^\veps  = \frac{1}{2L}(\theta^\veps-\eta^\veps),\\
\theta^\veps(0;\bx) = \theta^0(\bx)\convh \rho^\veps =: \theta^{0,\veps},\, \eta^\veps(0;\bx) = \eta^0(\bx)\convh \rho^\veps =: \eta^{0,\veps},
\end{align}
\end{subequations}
which enjoy the following estimates,
\begin{subequations}
\label{uniform_bounds_simplified_HM_eps}
\begin{align}
\sup_{t\in[0,T]}\norm{({ \theta ^\veps}(t), {\eta^\veps}(t))}_{\nLpP{p}} = \norm{(\theta^0, \eta^0)}_{\nLpP{p}},
\end{align}
for any $p\in[1,\infty]$,
\begin{align}
\sup_{t\in[0,T]}\norm{(\nabla_h{\theta^\veps}(t),\nabla_h{\eta^\veps}(t))}_{\nLzLphP{\infty}{4}} \leq {J_0(T)}\norm{(\nabla_h\theta^0,\nabla_h\eta^0)}_{\nLzLphP{\infty}{4}},
\end{align}
\begin{align}
\sup_{t\in[0,T]}\norm{(\partial_z{ \theta^\veps}(t), \partial_z{\eta^\veps}(t))}_{\nLpP{2}} \leq {H_0(T)}\norm{(\partial_z\theta^0,\partial_z\eta^0)}_{\nLpP{2}},
\end{align}
\begin{align}
&\sup_{t\in[0,T]}\norm{(\partial_t{ \theta^\veps}(t),\partial_t{ \eta^\veps}(t))}_{\nLpP{2}} \leq H_0(T)\norm{(\partial_z\theta^0,\partial_z\eta^0)}_{\nLpP{2}} \notag \\
&\qquad \qquad +  CLJ_0(T)\norm{(\theta^0,\eta^0)}_{\nLpP{\infty}}\norm{(\nabla_h\theta^0,\nabla_h\eta^0)}_{\nLzLphP{\infty}{4}},
\end{align}
where $J_0(T), H_0(T)$ are constants specified in \eqref{J_0_simplified_3D_HM} and \eqref{H_0_simplified_3D_HM}, respectively.
\end{subequations}

{\bf{Step 2}}: The convergence of the sequence $\{(\theta^\veps,\eta^\veps)\}_{\veps\in(0,\frac{L}{4}]}$, and passing in the limit to a weak solution $(\theta,\eta)$ as $\veps\rightarrow0$.
\begin{claim}
The sequence $\{(\theta^\veps,\eta^\veps)\}_{\veps\in(0,\frac{L}{4}]}$ is Cauchy in $C([0,T];\nLpP{2})$.
\end{claim}
{\it{Proof of Claim 4.4}}: Let $\frac{L}{4}>\veps_2>\veps_1>0$, and let $(\theta^{\veps_2},\eta^{\veps_2}), (\theta^{\veps_1},\eta^{\veps_1})$ be the corresponding weak solutions of system \eqref{eps_simplified_3D_HM}. By Corollary \ref{cont_dep_initial_data_simplified_HM} we have:
\begin{align*}
&\sup_{t\in[0,T]}\norm{(\theta^{\veps_1}(t)-\theta^{\veps_2}(t),\eta^{\veps_1}(t)-\eta^{\veps_2}(t))}_{\nLpP{2}}\leq \notag \\
&\qquad Ce^{\frac{C}{L^{1/2}}K_0(T) T\norm{(\nabla_h\theta^{0,\veps_2},\nabla_h\eta^{0,\veps_2})}_{\nLzLphP{\infty}{4}}}\norm{(\theta^{0,\veps_1}-\theta^{0,\veps_2},\eta^{0,\veps_1}-\eta^{0,\veps_2})}_{\nLpP{2}},
\end{align*}
where $J_0(T)$ is specified in \eqref{J_0_simplified_3D_HM}. Since $(\theta^0,\eta^0) \in \nLpP{2}$, then by Theorem 6 (\cite{Evans_1998}, page 630):
\begin{align*}
(\theta^{0,\veps}, \eta^{0,\veps}) \strong (\theta^0,\eta^0)  \qquad \quad \text {in}\,\, \nLpP{2},
\end{align*}
as $\veps \strong 0$. Thus, as $\veps_1, \veps_2 \strong 0$ , one can see that
\begin{align}
\sup_{t\in[0,T]}\norm{(\theta^{\veps_1}(t)-\theta^{\veps_2}(t), \eta^{\veps_1}(t)-\eta^{\veps_2}(t))}_{\nLpP{2}} \strong 0,
\end{align}
which implies that the sequence $\{(\theta^\veps,\eta^\veps)\}_{\veps\in(0,\frac{L}{4}]}$ is Cauchy in $C([0,T];\nLpP{2})$. Thus there exists $(\theta,\eta) \in C([0,T];\nLpP{2})$ such that:
\begin{align}
(\theta^\veps, \eta^\veps) \strong (\theta,\eta) \qquad \quad \text{in} \,\, C([0,T];\nLpP{2}),
\end{align}
as $\veps\strong 0$. Define $\omega = \frac{1}{2L}(\theta-\eta)$, $\bu = K\convh \omega$, then:

\begin{align}
\bu^\veps \strong \bu \qquad \quad \text{in} \,\, C([0,T];\LpP{2}),
\end{align}
as $\veps\strong 0$. Following the argument in Step 3 of the proof of Theorem \ref{global_existence_simplified_HM_classical}, one can show that $(\theta,\eta)$ satisfies \eqref{simplified_3D_HM} in the weak formulation. The uniform bounds \eqref{uniform_bounds_simplified_HM_eps} allow us to use the Banach Alaoglu compactness theorem and extract a weak-$\ast$ convergent subsequence, as in the proof of Theorem \ref{global_existence_simplified_HM_classical}, that inherits all the uniform bounds \eqref{uniform_bounds_simplified_HM_eps}. And so, $(\theta,\eta)$ is a weak solution of \eqref{simplified_3D_HM} and by the uniqueness of the limit, the weak-$\ast$ will be again $(\theta,\eta)$. Moreover, by Proposition \ref{nabla_h_u_Linfty_estimate_simplified_3D_HM} and by the argument introduced in Theorem \ref{global_existence_simplified_HM_classical},
\begin{align*}
&\sup_{t\in[0,T]}\norm{\nabla_h \bu(t)}_{\LpP{\infty}}\leq\\
&\frac{C}{L}\norm{(\theta^0,\eta^0)}_{\nLpP{\infty}}\log \left(4+CL^2J_0(T)\frac{\norm{(\nabla_h\theta^0,\nabla_h\eta^0)}_{\nLzLphP{\infty}{4}}}{\norm{(\theta^0,\eta^0)}_{\nLpP{2}}}\right).
\end{align*}
{\bf{Step 3}}:  Continuous dependence on the initial data and uniqueness of the solution.\\
By Corollary \ref{cont_dep_initial_data_simplified_HM},
\begin{align}
&\sup_{t\in[0,T]}\norm{(\theta_1(t)-\theta_2(t),\eta_1(t)-\eta_2(t))}_{\nLpP{2}}\notag \\
& \qquad\qquad \leq Ce^{\frac{C}{L^{1/2}}J_0(T)T \norm{(\nabla_h\theta_2^0,\nabla_h\eta_2^0)}_{\nLzLphP{\infty}{4}}}\norm{(\theta_1^0-\theta_2^0,\eta_1^0-\eta_2^0)}_{\nLpP{2}}.
\end{align}
Uniqueness is an immediate consequence of the above inequality. This completes the proof.
\end{proof}

Now, we consider the system:
\begin{subequations}\label{simplified_3D_HM_w_omega}
\begin{align}
\frac{\partial w}{\partial t} + ({\bold{u}} \cdot \nabla_h) w-U_0L\frac{\partial \omega}{\partial z} = 0, \qquad \frac{\partial \omega}{\partial t} +({\bold{u}}\cdot \nabla_h)\omega -\frac{U_0}{L}\frac{\partial w}{\partial z} = 0,\\
w = \frac{1}{2}(\theta+\eta), \qquad \omega = \frac{1}{2L}(\theta-\eta), \\
\nabla_h\times \bu = \omega , \qquad \nabla_h \cdot \bold{u} = 0.
\end{align}
\end{subequations}
Notice that for any space norm $\norm{.}$, we have:
\begin{align}\label{w_omega_theta_eta_estimate}
\norm{(w,L\omega)} \leq \norm{(\theta,\eta)} \leq 2\norm{(w,L\omega)}.
\end{align}
The next Theorem follows immediately from Theorem \ref{global_existence_uniqueness_simplified_HM} and \eqref{w_omega_theta_eta_estimate},
\begin{theorem}
\label{global_existence_uniqueness_HM_like}
Let $(w^0,\omega^0) \in \nLpP{\infty}$, $(\nabla_hw^0,\nabla_h\omega^0) \in \nLzLphP{\infty}{4}$, $(\partial_zw^0,\partial_z\omega^0) \in \nLpP{2}$, then the system \eqref{simplified_3D_HM_w_omega} has a unique weak solution $(w,\omega)$ in a sense similar to Definition \ref{weak_solution_definition_simplified_3D_HM}, such that:
\begin{align}
\sup_{t\in[0,T]}\norm{({w}(t), { L\omega}(t))}_{\nLpP{\infty}} \leq 2 \norm{(w^0, L\omega^0)}_{\nLpP{\infty}} ,
\end{align}
\begin{align}
\sup_{t\in[0,T]}\norm{(\nabla_h{w}(t),L\nabla_h{\omega}(t))}_{\nLzLphP{\infty}{4}} \leq 2{J_0(T)}\norm{(\nabla_hw^0,L\nabla_h\omega^0)}_{\nLzLphP{\infty}{4}},
\end{align}
\begin{align}
\sup_{t\in[0,T]}\norm{(\partial_z{ w}(t), L\partial_z{\omega}(t))}_{\nLpP{2}} \leq 2{H_0(T)}\norm{(\partial_zw^0,L\partial_z\omega^0)}_{\nLpP{2}},
\end{align}
\begin{align}
&\sup_{t\in[0,T]}\norm{(\partial_t{ w}(t),L\partial_t{ \omega}(t))}_{\nLpP{2}} \leq 2H_0(T)\norm{(\partial_zw^0,L\partial_z\omega^0)}_{\nLpP{2}} \notag \\
&\qquad \qquad +  4CL{J_0}(T)\norm{(w^0,L\omega^0)}_{\nLpP{\infty}}\norm{(\nabla_hw^0,L\nabla_h\omega^0)}_{\nLzLphP{\infty}{4}}.
\end{align}
Furthermore,
\begin{align}
&\sup_{t\in[0,T]}\norm{\nabla_h \bu(t)}_{\LpP{\infty}}\leq \notag \\
& 2\frac{C}{L}\norm{(w^0,	L\omega^0)}_{\nLpP{\infty}}\log \left(e^2+2CL^2J_0(T)\frac{\norm{(\nabla_hw^0,L\nabla_h\omega^0)}_{\nLzLphP{\infty}{4}}}{\norm{(w^0,L\omega^0)}_{\nLpP{2}}}\right),
\end{align}
where $\bu = K\convh{{\omega}}$. $C_0$ is a constant that depends on the domain $\nT^3$ and the norms of the initial data $(w^0,\omega^0)$ specified in \eqref{C_0_simplified_3D_HM}. $C$ is a positive dimensionless constant , $J_0(T), H_0(T)$ depends on $T$ and the on domain $\nT^3$ and the norms of the initial data $(w^0,\omega^0)$ specified in \eqref{J_0_simplified_3D_HM} and  \eqref{H_0_simplified_3D_HM}, respectively.
Moreover, if $(w_0^1,\omega_0^1), (w_0^2,\omega_0^2) \in \nLpP{\infty}$, $(\nabla_hw_1^0,\nabla_h\omega_1^0)$, $(\nabla_hw_2^0,\nabla_h\omega_2^0) \in \nLzLphP{\infty}{4}$, $(\partial_zw^0_1,\partial_z\omega_1^0)$, $(\partial_zw_2^0,\partial_z\omega_2^0) \in \nLpP{2}$, then the corresponding weak solutions $(w_1,\omega_1)$, $(w_2,\omega_2)$ of system \eqref{simplified_3D_HM} satisfy:
\begin{align}
&\sup_{t\in[0,T]}\norm{(w_1(t)-w_2(t),L(\omega_1(t)-\omega_2(t)))}_{\nLpP{2}}\leq\notag \\
&\qquad 2Ce^{2\frac{C}{L^{1/2}}J_{0}(T)T \norm{(\nabla_hw_2^0,L\nabla_h\omega_2^0)}_{\nLzLphP{\infty}{4}}}\norm{(w_1^0-w_2^0,L(\omega_1^0-\omega_2^0))}_{\nLpP{2}}.
\end{align}
\end{theorem}
\bigskip

\section*{Acknowledgements}
A.F. would like to thank the Faculty of Mathematics and Computer Science
at the Weizmann Institute of Science for its
kind hospitality where part this work was completed.
This work was supported in part by the NSF grants DMS-1009950,
DMS-1109640, DMS-1109645 and DMS-1109022. E.S.T.  also
acknowledges the support of the Alexander von Humboldt
Stiftung/Foundation and the Minerva Stiftung/Foundation.

\bigskip



\end{document}